\documentclass[10pt, article]{amsart}
\subjclass[2010]{57N65, 55N05, 55N35, 55Q05, 53C22(primary), and 30F99, 53B21(secondary)}

\usepackage{tikz}
\usetikzlibrary{shadings}
\usepackage{pgfplots}
\usetikzlibrary{calc}
\usetikzlibrary{arrows}
\newcommand{\midarrow}{\tikz \draw[-triangle 90] (0,0) -- +(.1,0);}

\usepackage{ae} 
\usepackage[T1]{fontenc}
\usepackage[cp1250]{inputenc}
\usepackage{amsmath}
\usepackage{amssymb, amsfonts,amscd,verbatim}

\usepackage{indentfirst}
\usepackage{latexsym}
\input xy
\xyoption{all}

\newcommand{\new}[1]{\textcolor{black}{#1}} 

\usepackage{amsmath}    

\theoremstyle{plain}
\newtheorem{Prop}{Proposition}[section]
\newtheorem{Proposition}[Prop]{Proposition}

\newtheorem{Theorem}[Prop]{Theorem}
\newtheorem{Cor}[Prop]{Corollary}
\newtheorem{Corollary}[Prop]{Corollary}

\newtheorem{Lemma}[Prop]{Lemma}

\theoremstyle{definition}
\newtheorem{Def}[Prop]{Definition}
\newtheorem{Definition}[Prop]{Definition}

\theoremstyle{remark}

\newtheorem{Remark}[Prop]{Remark}

\newtheorem{Example}[Prop]{\bf Example}

\newcommand{\diam}{\operatorname{Diam}\nolimits}

\newcommand{\Rips}{\operatorname{Rips}\nolimits}
\newcommand{\cRips}{\operatorname{\overline{Rips}}\nolimits}
\newcommand{\Cech}{\operatorname{Cech}\nolimits}
\newcommand{\cCech}{\operatorname{\overline{Cech}}\nolimits}

\def\RR{{\mathbb R}}
\def\FF{{\mathbb F}}

\def\II{{\mathbb I}}

\def\ZZ{{\mathbb Z}}
\def\NN{{\mathbb N}}
\def\P{{\mathcal P}}
\def\C{{\mathcal C}}

\def\S{{\mathcal S}}
\def\D{{\mathcal D}}
\def\L{{\mathcal L}}
\def\cB{{\overline {B}}}
\def\f{{\varphi}}
\def\eps{{\varepsilon}}
\def\s{\sigma}

 \def\size{\mathop{\textrm{size}}\nolimits}
  \def\crit{\mathop{\textrm{crit}}\nolimits}
\def\length{\mathop{\textrm{length}}\nolimits}
\def\Rcrit{\mathop{\textrm{Rcrit}}\nolimits}
\def\arginf{\mathop{\textrm{arginf}}}

\numberwithin{equation}{section}
\title[1-Dimensional Intrinsic Persistence of geodesic spaces]
{1-Dimensional Intrinsic Persistence of geodesic spaces}

\author{\v Ziga ~Virk}
\address{University of Ljubljana}
\email{zigavirk@gmail.com}

\begin{document}

\thanks{This research was conducted while the author was a postdoctoral researcher in the group of prof. Edelsbrunner at IST Austria.}
\thanks{The author would like to thank the referees for careful reading and helpful suggestions.}

\maketitle
\begin{center}
\today
\end{center}

\begin{abstract}
Given a compact geodesic space $X$, we apply the fundamental group and alternatively the first homology group functor to the corresponding Rips or \v Cech filtration of $X$ to obtain what we call a persistence object. This paper contains the theory describing such persistence: properties of the set of critical points, their precise relationship to the size of holes, the structure of persistence and the relationship between open and closed, Rips and \v Cech induced persistences. Amongst other results we prove that a Rips critical point $c$ corresponds to an isometrically embedded circle of length $3c$, that a homology persistence of a locally contractible space with coefficients in a field encodes the lengths of the lexicographically smallest base, and that Rips and \v Cech induced persistences are isomorphic up to a factor $3/4$. The theory describes geometric properties of the underlying space encoded and extractable from persistence.
\end{abstract}

\tableofcontents
\section{Introduction}

The field of Topological Data Analysis (TDA) is  in a large part based on the concept of an approximation or discretization of a metric space by a complex of choice, Rips or \v Cech, open or closed. In order to avoid an artificial choice of the scale parameter the complexes are constructed for all scales and connected by natural inclusions, together forming a filtration of a metric space. To obtain interesting topological summaries one  computes the corresponding persistent homology using coefficients from a preferred field for computational reasons. The main pipeline of TDA along these lines follows two steps:
\begin{enumerate}
 \item Given a metric space, construct a filtration using some  construction of a complex. Any construction should roughly provide the same information.
 \item For a chosen field compute persistent homology, which roughly represents the size of holes in the space. 
\end{enumerate}
While these two steps are up to some point theoretically justified in the general case by the Rips/\v Cech interleaving and the Nerve theorem, there is nonetheless a lack of a precise formulation in any setting  which would explain how the result depends on a choice of a complex or how the sizes of holes are measured precisely (with a notable exception \cite{EP}). Consequently there is a  push in the research community to develop statistical framework for the interpretation of the obtained persistence diagrams (PDs), which would allow applications of the methods of artificial intelligence.

With this paper we are initiating a  research program of  intrinsic persistence using a different approach. Our goal is not just to provide a theoretical background of TDA but to build a general theory of approximations of geodesic spaces by complexes/filtrations of which the persistent homology via geodesic metric is a special case. Our primary setting is a filtration of a geodesic metric space, which covers the most interesting families: compact Riemannian manifolds, simplicial complexes, Cayley graphs and Peano continua. It can be thought of as a multiscale discretization of a space. In this paper we develop the corresponding theory of $1$-dimensional persistence, which refers to the persistence obtained from any filtration by applying the fundamental group or the homology group (with any coefficients) functor. This expands the standard notion of persistence. Besides applications to TDA, such generalization allows us to build a theory whose results refer to other settings in which similar approximations are considered.  These include the Shape Theory, which studies the limit of filtrations when the scale parameter goes to zero, and Coarse Geometry (in particular Anti-Cech approximations of Dranishnikov and results of \cite{Comb}) where the scale parameter goes to infinity. In the later case, the results of this paper provide a precise interpretation for coarse simple connectedness in terms of the lengths of relations in a Cayley graph of a group. Furthermore, the results imply that persistence  encodes geometric information of the underlying space and should also represent an interesting tool to geometry. More details are provided in the Future Work Section \ref{SectFW}.  

The core thread of the theory introduced in this the paper is developed using open Rips filtration and the corresponding fundamental group functor for a compact geodesic space $X$. Parallel to it we will be stating results for the corresponding homology. In the end we deduce equivalent results for \v Cech and close filtrations. The main approach is the interaction between a discrete setting modelled by Rips complexes, and a continuous setting of a geodesic space. This interaction is set up by Definition \ref{DefRLoop} and Proposition \ref{PropRips}, and studies throughout the paper.
The main new results developed by this paper are roughly the following:
\begin{enumerate}
 \item Theorems \ref{ThmPerBasis} and \ref{ThmIntrLoc}: a Rips-critical point $c$ of persistence corresponds to an isometrically embedded circle(s) of length $3c$, which arises from the boundaries of critical triangles;
 \item Theorem \ref{PropFin} and Corollary \ref{CorFinCritValue}: $0$ is the only possible accumulation point of the set of critical points, with the later being finite for locally contractible spaces;
 \item Theorems \ref{ThmPerSpa}, \ref{ThmPerCirc}, and \ref{ThmPerDiam}: persistence measures precisely the 'size' of holes measured by the length (equivalently the diameter or the radius of the smallest enclosing disc) of the corresponding embedded circle;
 \item Remark \ref{RemTame} and Theorem \ref{ThmLexicograph}: persistence via homology with coefficients in a field $\FF$ is q-tame \cite{Cha1} and in a locally contractible case encodes precisely the lengths of lexicographically smallest generating set of $H_1(X,\FF)$, see Figure \ref{FigExample};
 \item Theorem \ref{ThmRipsCech}: persistences via Rips and \v Cech filtrations are isomorphic up to a factor $3/4$.
\end{enumerate}

Besides these results, we obtain a wealth of other results on persistence, topology, and geometry of $X$, including five equivalent formulation of persistence via open Rips complexes. Using the obtained relationship to closed and \v Cech filtrations we may consequently express persistence in 12 different ways! Below is a sample result using notation introduced in Section \ref{SectPre}:

\begin{Theorem}
 Suppose $X$ is a compact Riemannian manifold, $\FF$ is a field and $l_1, l_2, \ldots, l_n$ are the lengths of the shortest generating set of representatives of $H_1(X, \FF)$. Then the following persistences are isomorphic:
\begin{enumerate}
 \item $\big\{H_1(\Rips(X,4r), \FF) \}_{r>0}$
 \item $\big\{H_1(\Cech(X,3r),\FF)\big\}_{r>0}$
 \item $\{H_1(X, \FF)/\S(X, 6r)\}_{r>0}$
  \item $\big\{H_1(X, \FF)/\L(X, 12r)\big\}_{r>0}$
   \item $\{H_1(X, \FF)/\D(X, 6r)\}_{r>0}$
   \item $\bigoplus_{i=1}^n (\II_{(0, 12 l_i]})_r$, i.e., the sum of interval  modules for all intervals $(0, 12 l_i]$,
   \end{enumerate}
   where $\S(X,a), \L(X, a), \D(X,a)$ are subspaces of $H_1(X, \FF)$ generated by loops size less than $a$, with the size being measured as the radius of smallest enclosing ball, length and the diameter of the loop correspondingly.
\end{Theorem}

Intrinsic $1$-dimensional persistence was, to the best of our knowledge, only determined in two cases: for metric graphs via open \v Cech filtration using homology with coefficients in a field \cite{7A}; and for $S^1$ in \cite{AA}, where the authors actually provide a complete description  in all dimensions. The results of the later paper  complement the theory of this paper: on one hand its 1-dimensional case coincides with the easiest non-trivial example for our theory; on the other hand it may be combined with insight (1) above (more about this idea in Section \ref{SectFW}) to extract lengths of some geodesics from higher-dimensional persistent homology and give a computational framework for the critical points of the fundamental group persistence. 
An application of our theory to metric graphs also provides a generalization of the main result of \cite{7A}. As for computational implementation, the intrinsic distace was approximated in \cite{Dey} although the motivation for doing so was different. \cite{ZV1} is a successor to this paper described in Section \ref{SectFW}. 

The structure of the paper is the following. In Section \ref{SectRipsCxes} we set up the interaction between the discrete setting of Rips complexes and the continuous setting of a geodesic space via samples and fillings. We then prove basic geometric statements for this interaction and Rips complexes. In Section \ref{SecSizeHoles} we describe sizes of holes detected by this interaction. These are used in the correspondence theorems of Section \ref{SectFundGroup}, which relate the fundamental groups of Rips complexes to the fundamental groups of the underlying space modulo 'small' holes. Section \ref{SectLimit} provides the limiting arguments which allow us to obtain interesting loops in the limit. Results of all these sections are combined in Section \ref{SectStruc}, where we show that critical values are generated by geodesic circles. Section \ref{SectHStruc} provides homological versions of results of  Section \ref{SectStruc} and the connection between PD and the shortest generating set.
In Section \ref{SectExample} we give a short illustrative example. In Section \ref{SectCech} we prove a connection to \v Cech complexes and in Section \ref{SectClosed} we consider other generalizations.  Section \ref{SectFW} outlines future work.

\section{Preliminaries}
\label{SectPre}
In this section we introduce the concepts and notions used later. For details and undefined notions of Algebraic Topology and TDA see \cite{Hat} and \cite{EH} respectively.

Suppose $(X,d)$ is a metric space. For $A\subset X$ we define diameter as $\diam(A)=\sup_{x,y\in A}d(x,y)$. 
For $x\in  X$ and $r\geq 0$, $B(x,r)$ represents the open ball and $\cB(x,r)$ represents the closed ball.
We will focus on \textbf{geodesic}  spaces, i.e., on metric spaces $(X,d)$ satisfying the following condition: for each $x,y\in X$ there exists a path, called a \textbf{geodesic}, from $x$ to $y$ of length $d(x,y)$. Equivalently, $X$ is geodesic if for each $x,y\in X$ there exists an isometric embedding of $[0, d(x,y)]$ into $X$ with $0\mapsto x$ and $ d(x,y)\mapsto y$. All our spaces will be geodesic, unless explicitly stated otherwise. A geodesic between two points may not be unique. The class of geodesic metric spaces includes all Riemannian manifolds. A \textbf{geodesic circle} in a metric space $X$ is an isometrically embedded circle (equipped with the geodesic metric) of positive circumference. When necessary we will consider $\bullet\in X$ to be the basepoint of $X$. Loops in a space will be considered either as a map to $X$ or a subset of $X$, depending on the context. Given a loop $\alpha$ in $X$, the induced homotopy class (if $\alpha$ is based we consider based homotopy class) is referred to as $[\alpha]_{\pi_1}$ and the induced cycle in $H_1(X,G)$ is referred to as $[\alpha]_G$.  A loop in $X$ (based or unbased) is \textbf{essential} if it is not contractible. 
The concatenation of loops or paths $\alpha$ and $\beta$ is denoted by $\alpha * \beta$ (in the later case we naturally require that the endpoint of $\alpha$ is the initial point of $\beta$). Given a path $\alpha \colon [0,a]\to X$ we define the inverse path $\alpha^-\colon [0,a]\to X$  by $\alpha^-(t)=\alpha(a-t)$. 

\begin{Def}\label{DefComplexes}
Given  $r>0$ we define various complexes with vertex set $
X$. The condition besides the name determines when a finite subset $\s\subset X$ belongs to the complex.
\begin{enumerate}
 \item \textbf{(Open) Rips} (or Vietoris-Rips) \textbf{complex} $\Rips(X,r): Diam(\s) < r$.
 \item \textbf{Closed Rips complex} $\cRips(X,r): Diam(\s) \leq r$.
 \item \textbf{(Open) \v Cech complex} $\Cech(X,r): \cap_{z\in\s} B(z,r) \neq \emptyset$.
  \item \textbf{Closed \v Cech complex} $\cCech(X,r): \cap_{z\in\s} \cB(z,r) \neq \emptyset$.
\end{enumerate}
\end{Def}

We will call $1$-dimensional simplices edges, and $2$-dimensional simplices triangles. For a simplex $\sigma$ in a \v Cech complex we refer to any point $w$ of $\cap_{z\in\s} B(z,r)$ as a witness of $\sigma$, or say that $w$ witnesses $\sigma$.

\begin{Def}\label{DefFiltration}
 Let $\C$ denote any of the complexes mentioned in Definition \ref{DefComplexes}. The \textbf{induced filtration} of $X$ is the collection $\{\C(X,r)\}_{r>0; r\in\RR}$ along with naturally induced simplicial \textbf{bonding} maps $i_{p,q}\colon \C(X,p)\to \C(X,q)$ for all $p<q$, which are the identity on vertices. Considering this induced filtration as a category, we will apply various functors $\P$ to obtain (induced) $\P$-filtration of $X$. We will refer to this filtration as \textbf{persistence}. 
 For example, the $\pi_1$-persistence of $(X, \bullet)$ via open Rips complexes refers to a collection $\{\pi_1(\Rips(X,r), \bullet)\}_{r>0}$ along with homomorphisms $\pi_1(i_{p,q})\colon \pi_1(\Rips(X,p),\bullet)\to \pi_1(\Rips(X,q), \bullet)$ for all $p<q$, which are induced by the natural inclusions. We will usually omit the word 'induced'. If $\FF$ is a field then $H_1(\_, \FF)$-persistence consists of vector spaces and is usually referred to as a persistence module. An \textbf{isomorphism} $f$ between persistences $\{A_r\}_{r>0}$ and $\{B_r\}_{r>0}$ is a collection of isomorphisms $f_r \colon A_r \to B_r$ which commutes with the corresponding bonding maps. The isomorphism is denoted by $\{A_r\}_{r>0}\cong \{B_r\}_{r>0}$. 
 \end{Def}
 
Suppose  that $\FF$ is a field and that for all $r>0$,  vector spaces $H_1(\Rips(X, r), \FF)$ are finitely generated. In this case the corresponding  persistence module decomposes as a sum of elementary intervals  $\II_{\langle b,d \rangle}$, each of which is  of the form
$$
(\II_{\langle b,d \rangle})_r = 
\begin{cases}
 \FF, \quad r\in \langle b,d \rangle \\
 0, \quad r\notin \langle b,d \rangle,
\end{cases}
$$ 
with bonding maps being surjective (identities or trivial) and with $ \langle b,d \rangle$ denoting an interval (with either endpoint being open or closed).
 A \textbf{persistence diagram} $PD(X,\FF)$ is a set consisting of points $(b,d)\in \RR^2, (b<d)$, which appear as the endpoints of the above mentioned intervals ($b$ is referred to as the birth, and $d$ as the death of an interval). If we want to indicate whether an interval is open or closed at a specific endpoint we may use decorated points as in \cite{Cha1}. To each point we attach a degree indicating the number of intervals with the corresponding endpoints.

Given $\delta>0$, a $\delta$-\textbf{interleaving} between persistences $\{A_r\}_{r>0}$ and $\{B_r\}_{r>0}$ is a collection of homomorphisms $f_r \colon A_r \to B_{r+\delta}$ and $g_r \colon B_r \to A_{r+\delta}$ which commutes with each other and with the corresponding bonding maps.

Due to technical, interpretational, and computational difficulties, persistent homotopy groups (mentioned for example in \cite{Le}) never really gained a traction comparable to that of persistent homology.  
One of the main goals of this paper is to provide a very precise interpretation of these invariants in terms of closed geodesics, which should enhance the appreciation of intrinsic persistence,   provide theoretical background, and hopefully  lead to a  computational implementation. In the forthcoming paper we will apply this understanding to higher dimensional persistence and show that a lot of information (i.e.,  critical points, see Definition \ref{DefCritical}) about $\pi_1$-persistence can be extracted from higher dimensional homology persistence.

\begin{Def}\label{DefCritical}
Given a $\P$-filtration $\{C_r\}_{r>0}$, $p>0$ is a:
\begin{enumerate}
 \item \textbf{left critical value}, if for all small enough $\eps>0$ the map $i_{p-\eps,p}$ is not an isomorphism (in the category determined by $\P$);
  \item \textbf{right critical value}, if for all  $\eps>0$, the map $i_{p,p+\eps}$ is not an isomorphism;
\end{enumerate}
We call $p>0$ a \textbf{critical value}, if it is either of the above.
\end{Def}

We will later show that when considering intrinsic metric of compact geodesic spaces, open filtrations (that is, filtrations by open complexes mentioned in Definition \ref{DefComplexes}) admit only right critical values for $\pi_1$ and $H_1$ persistence (Corollary \ref{CorDirCrit}), and closed filtrations admit only left critical values for $\pi_1$ and $H_1$ persistence (\textbf{c} of Section \ref{SectClosed}). The later statement is not true for higher-dimensional invariants, which can be observed from the classification of the closed filtrations of the circle as proved in \cite{AA}.

\begin{Def}
 \label{DefSLSC}
A path-connected space $X$ is \textbf{semi-locally simply-connected} (SLSC) if for every point $x\in X$ there exists a neighborhood $U$ of $x$ for which the image of the inclusion induced map $\pi_1(U,x)\to \pi_1(X,x)$ is trivial.

Given an Abelian group $G$, a path-connected space $X$ is $G$-\textbf{semi-locally simply-connected} ($G$-SLSC) if for every point $x\in X$ there exists a neighborhood $U$ of $x$ for which the image of the inclusion induced map $H_1(U,G)\to H_1(X,G)$ is trivial. Property $G$-SLSC is weaker than SLSC.
\end{Def}

The property of being SLSC essentially means that locally, there are no essential loops at any point. Putting it differently, every point has a small enough neighborhood which contains no essential loop in $X$. SLSC spaces include all locally contractible spaces such as manifolds, complexes, etc. A standard example of a non-SLSC space is the Hawaiian Earring $HE$:  a planar union of circles of radius $1/n$ for $n\in \NN$, sharing a common point.

\begin{Example}\label{ExampleBasic}
The simplest non-trivial example of a non-contractible geodesic space is probably a circle $C$. For simplicity we assume it is of circumference $1$. Its $1$-dimensional persistence is a good example on which to build the intuition for the results to come. The complete persistence of $C$ was computed in \cite{AA}.

It turns out that for $r \leq 1/3$,  $\Rips(C,r)$ is homotopy equivalent to a circle, with bonding maps being homotopy equivalences. For $r> 1/3$, the complex $\Rips(C,r)$ is simply connected (in fact, it is homotopy equivalent to various wedges of higher dimensional spheres  or contractible \cite{AA}). Thus $1/3$ is a right critical value. Intuitively we can think of the critical value $1/3$ as the infimum of values $r$, at which $\Rips(C,r)$ contains a triangle whose convex hull contains the center of $C$, if we imagine $C$ embedded in the plane in the usual way. For each $r>1/3$ any three equidistant points on $C$ provide such configuration for $\Rips(C,r)$, as depicted in Figure \ref{FigAA}. In the case of $\cRips(C,r)$ the value $1/3$ is a left critical value and three equidistant points on Figure \ref{FigAA} induce a simplex in $\cRips(C,1/3)$ required to collapse all loops.

A similar intuition works for the \v Cech complex. It turns out that for $r \leq 1/4$,  $\Cech(C,r)$ is homotopy equivalent to a circle, with bonding maps being homotopy equivalences. For $r> 1/4$, the complex $\Cech(C,r)$ is simply connected (in fact, it is again homotopy equivalent to various wedges of higher dimensional spheres  or contractible \cite{AA}). An equivalent treatment also refers to closed \v Cech complexes. A configuration of three points, whose convex hull contains the center of $C$, is in this case provided by two opposite points on $C$ and any of their midpoints, as depicted on Figure \ref{FigAA}. In this case the intersection of closed balls of radius $1/4$ around these points is the mentioned midpoint.
\end{Example}

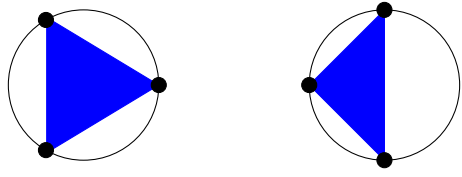
\begin{figure}
\begin{tikzpicture}
\coordinate (A) at (-1,0);
\coordinate (B) at (-2.49,.9);
\coordinate (C) at (-2.49,-.9);
\draw  (-2,0) circle [radius=1];
\draw  (2,0) circle [radius=1];
\draw [blue, fill=blue] (1,0) -- (2,1) -- (2,-1) -- (1,0);
\draw [blue, fill=blue] (A) --(B) -- (C)-- (A);
\draw [fill] (-2,0)+(120:1) circle [radius=0.1];
\draw [fill] (-2,0)+(240:1) circle [radius=0.1];
\draw [fill] (-2,0)+(0:1) circle [radius=0.1];
\draw [fill] (2,-1) circle [radius=0.1];
\draw [fill] (2,1) circle [radius=0.1];
\draw [fill] (1,0) circle [radius=0.1];
\end{tikzpicture}
\caption{The minimal (with respect to the radius parameter of the complex) configuration of three points on a circle forming a simplex in Rips (left) and \v Cech (right) complex, whose convex hull contains the center of the circle.}
\label{FigAA}
\end{figure}

\section{Geometry of Rips complexes}
\label{SectRipsCxes}

In this section we present the foundations for the interaction between the discrete setting modelled by Rips complexes, and the continuous setting of geodesic spaces. The interaction is utilised by fillings and samples of Definition \ref{DefRLoop}, which induce crucial maps in  later sections. 

\begin{Def}
\label{DefRLoop}
 Fixing $r>0$ we define the following notation:
 \begin{enumerate}
 	\item $r$-\textbf{loop} $L$: a simplicial loop in $\Rips(X,r)$ considered as a sequence of points $(x_0, x_1, \ldots, x_k, x_{k+1}=x_0)$ in $X$ with $d(x_i, x_{i+1})<r, \forall i\in \{0,1,\ldots, k\}$.   An $r$-loop is $r$-\textbf{null} if it is contractible in $\Rips(X,r)$;
	\item \textbf{filling} of $L$: any loop in $X$ obtained from $L$ by connecting $x_i$ to $x_{i+1}$ by a geodesic for all $i\in \{0,1,\ldots, k\}$;
	\item $\size(L)=|L|=k+1$;
	\item $\length(L)=\sum_{i=0}^k d(x_i, x_{i+1})$, coincides with the length of any filling of $L$;
	\item $r$-\textbf{sample} of a loop $\alpha \colon [0,a]\to X$: a choice of $0\leq t_0<t_1<\ldots < t_m \leq a$ with $\diam \alpha ([t_{i},t_{i+1}])<r, \forall i\in \{0,1,\ldots, m-1\}$ and $\diam (\alpha([0,t_0] \cup[t_m,a]) <r$. Such choice of $t_i$'s exists by compactness. By an $r$-sample we will usually consider the induced $r$-loop $(\alpha(t_0), \alpha(t_1), \ldots, \alpha(t_m), \alpha(t_{0}))$. If $\alpha$ is based at $\bullet$, we will assume $t_0=0$;
	\item \textbf{Rips critical value} $\Rcrit (\alpha)$ of a loop $\alpha \colon [0,a]\to X$: $$\arginf_r \{r\textit{-sample of }\alpha \textit{ is } r\textit{-null}\},$$ i.e., the infimum of the set of all possible values of $r$ for which an $r$-sample of $\alpha$ is $r-$null. By Proposition \ref{PropRips}, the definition does not depend on the choice of an $r$-sample.
 \end{enumerate}
 
 Two $r$-loops are $r$-\textbf{homotopic}, if they are homotopic in $\Rips(X,r)$. The corresponding simplicial homotopy in $\Rips(X,r)$ is referred to as $r$-\textbf{homotopy}. If the second $r$-loop is constant we also call it $r$-\textbf{nullhomotopy}. Depending on the context we may be considering based or unbased homotopies. The \textbf{concatenation} $L * L'$ of $r$-loops $L$ and $L'$ is defined in the obvious way by joining (concatenating) the defining sequences. Note that a filling of the join is the join of fillings of $r$-loops.
\end{Def}

 Note that if loop $\alpha \colon [0,a]\to X$  in Definition \ref{DefRLoop} is a parameterization by the natural parameter, then $\alpha$ is of finite length. This condition excludes non-rectifiable (i.e., loops of infinite length) loops. Homotopy classes of loops whose each representative is of infinite length naturally appear, for example, in $HE$ by circumventing infinitely many circles in an appropriate fashion. We will later show that in the case of SLSC spaces this phenomenon is absent, i.e., each homotopy class of a loop has a rectifiable representative. 

The following proposition presents basic properties of $r$-loops, which will be used generously throughout the paper.

\begin{Prop}
\label{PropRips}
 Let $X$ be geodesic and fix $0<r<r'$. Then the following hold:
 \begin{enumerate}
 	\item if $L$ is an $r$-loop then it is an $r'$-loop as well;
	\item if $r$-loop $L$ is $r$-null then it is $r'$-null as well;
	\item any $r$-loop of size $3$ is $r$-null;
	\item given a loop $\alpha \colon [0,a]\to X$, any two $r$-samples of $\alpha$ are $r$-homotopic;
	\item any $r$-sample of a loop of length less than $3r$ is $r$-null;
	\item choose loops $\alpha \colon [0,a]\to X$ and $\alpha' \colon [0,a']\to X$ and take any two of their $r$-samples $L$ and $L'$. If $\alpha$ and $\alpha'$ are  homotopic, then  $L$ and $L'$ are $r$-homotopic (the statement holds for both based and unbased versions). If $G$ is an Abelian group and $[\alpha]_G=[\alpha']_G\in H_1(X,G)$ then $[L]_G=[L']_G\in H_1(\Rips(X,r),G)$; 
	\item if a loop $\alpha \colon [0,a]\to X$ is contractible, then any of its $r$-samples is $r$-null;
	\item suppose two $r$-loops, $L$ and $L'$, are given by $(x_0, x_1, \ldots, x_k, x_{k+1}=x_0)$ and $(y_0, y_1, \ldots, y_k, y_{k+1}=y_0)$. If
	$$
	\max_{i\in \{0,1,\ldots, k\}}d(x_i,y_i) < r-\max_{i\in \{0,1,\ldots, k\}} \{d(x_i,x_{i+1}),d(y_i, y_{i+1})\},
	$$
	then $L$ and $L'$ are $r$-homotopic (the statement holds for both based and unbased versions);
	\item maps $\pi_1(i_{r,r'})\colon \pi_1(\Rips(X,r),\bullet)\to \pi_1(\Rips(X,r'), \bullet)$ and their homological counterparts are surjective.
\end{enumerate}
\end{Prop}

\begin{proof}
 (1) and (2) and (3) are trivial. In order to prove (4) consider two $r$-samples $0 \leq t_0<t_1<\ldots < t_m \leq a$ and $0 \leq t'_0<t'_1<\ldots < t'_{m'} \leq a$. We claim that both are $r$-homotopic to the $r$-sample determined by the ordering of the collection $\{t_i\}_{i=0}^m \cup \{t'_i\}_{i=0}^{m'}$ (the later is obviously an $r$-sample). Using the inductive argument it suffices to prove the following claim: given an $r$-sample $L$ determined by $0 \leq s_0<s_1<\ldots < s_{k} \leq a$ and $p\in[0,a]\setminus \{s_i\}_{i=0}^k$, the $r$-sample $L'$ obtained by adding $p$ to the collection  $\{s_i\}_{i=0}^k$ is $r$-homotopic to $L$. In order to prove the claim consider two cases:
\begin{itemize}
	\item  if $p\in (s_i,s_{i+1})$ for some $i$ then the $r$-homotopy required by the claim is given by the triangle $[\alpha(s_i), \alpha(s_{i+1}), \alpha(p)]$ in $\Rips(X,r)$;
	\item else the triangle $[\alpha(s_0), \alpha(s_{k}), \alpha(p)]$ in $\Rips(X,r)$ does the job.
\end{itemize}
 
 (5) By (4) it suffices to provide a proof for any $r$-sample. We consider an $r$-sample consisting of three equidistant points. Such triple obviously forms a triangle in $\Rips(X,r)$ hence is $r$-null.
 
 (6) We will provide a proof for unbased homotopies, the based version is essentially the same. Using (4) again we will only provide $r$-homotopy between two specific $r$-samples. Consider a homotopy $H\colon S^1 \times I \to X$ between loops $\alpha=H|_{S^1\times \{0\}}$ and $\alpha'=H|_{S^1\times \{1\}}$, where we reparameterize the loops so that their domain is $S^1$. Choose a triangulation $\Delta$ of $S^1\times I$ so that the image of each $2$-simplex by $H$ is of diameter less than $r$: we can do that by taking a triangulation  subordinate to the open cover $\big\{H^{-1}(B(x,r/2))\big\}_{x\in X}$. This triangulation induces the required $r$-homotopy $H' \colon (S^1 \times I, \Delta)\to \Rips(X,r)$: each $2$-simplex determined by points $(a_1,t_1),(a_2,t_2),(a_3, t_3)$ in $S^1\times I$ is linearly mapped to the $2$-simplex $[H(a_1,t_1), H(a_2,t_2), H(a_3, t_3)]$ in $\Rips(X,r)$. 
 
The homological version goes along the same lines. Suppose $[\alpha]_g-[\alpha']_G = \partial \sum_{i=1}^k \lambda_i T_i$, where
 $\lambda_i\in G$ and $T_i$ is a singular $2$-simplex in $X$ for each $i$. As before, a fine enough triangulation of each $T_i$ (formally we take triangulations of the domains of $T_i$, which are viewed as maps from the standard $2$-simplex into $X$) decomposes $T_i$ into a sum of $2$-simplices, each of which is contained in a ball of radius $r/2$. Adding these simplices multiplied with $\lambda_ i$ and summing over all $i$ we prove the claim.
  (7) follows from (6). 
 
 (8) The proof is the same for the based and unbased version. The $r$-homotopy is provided by the obvious arrangement of $2$-simplices $[x_i,x_{i+1},y_i]$ and $[y_i,y_{i+1},x_{i+1}]$ as depicted on Figure \ref{FigOpenRips}, where index $i\in \{0,1,\ldots,k\}$ is considered $mod(k+1)$. The condition of the statement guarantees that the mentioned $2$-simplices are all in $\Rips(X,r)$.
 
 (9) Take a $q$-loop $L$. By (6) any $p$-sample $L'$ of a filling of $L$ is $q$-homotopic to $L$. Hence $[L']_{\pi_1}\in   \pi_1(\Rips(X,p),\bullet)$ is the preimage of  $[L]_{\pi_1}\in \pi_1(\Rips(X,q), \bullet)$ by $i_{p,q}$. 
 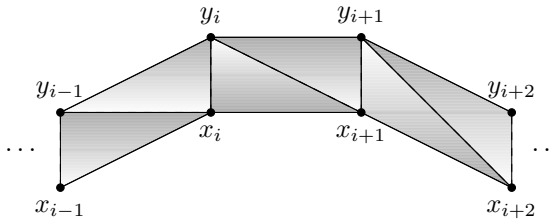
\begin{figure}
\begin{tikzpicture}
\draw[top color=gray,bottom color=white, fill opacity =.5] (-2,1) -- (0,2) --(0,1) -- cycle; 
\draw[top color=gray,bottom color=white, fill opacity =.5] (-2,1) -- (-2,0) --(0,1) -- cycle; 
\draw[bottom color=gray,top color=white, fill opacity =.5] (2,1) -- (0,2) --(0,1) -- cycle; 
\draw[top color=gray,bottom color=white, fill opacity =.5] (2,1) -- (0,2) --(2,2) -- cycle; 
\draw[bottom color=gray,top color=white, fill opacity =.5] (2,1) -- (4,0) --(2,2) -- cycle; 
\draw[top color=gray,bottom color=white, fill opacity =.5] (4,1) -- (4,0) --(2,2) -- cycle; 
\node [circle,  scale=.3, fill=black,draw, label=above:$y_{i-1}$] (a1) at (-2,1) {};
\node [circle,  scale=.3, fill=black,draw, label=above:$y_{i}$] (a2) at (0,2) {};
\node [circle,  scale=.3, fill=black,draw, label=above:$y_{i+1}$] (a3) at (2,2) {};
\node [circle,  scale=.3, fill=black,draw, label=above:$y_{i+2}$] (a4) at (4,1) {};
\draw (a1)--(a2)--(a3)--(a4);
\node [circle,  scale=.3, fill=black,draw, label=below:$x_{i-1}$] (b1) at (-2,0) {};
\node [circle,  scale=.3, fill=black,draw, label=below:$x_{i}$] (b2) at (0,1) {};
\node [circle,  scale=.3, fill=black,draw, label=below:$x_{i+1}$] (b3) at (2,1) {};
\node [circle,  scale=.3, fill=black,draw, label=below:$x_{i+2}$] (b4) at (4,0) {};
\draw (b1)--(b2)--(b3)--(b4);
\draw [dashed] (a1)--(b1);
\draw [dashed] (a2)--(b2);
\draw [dashed] (a3)--(b3);
\draw [dashed] (a4)--(b4);
\draw [dashed] (a1)--(b2);
\draw [dashed] (a2)--(b3);
\draw [dashed] (a3)--(b4);
\node (a5) at (-2.5,.5) {$\dots$};
\node (a6) at (4.5,.5) {$\dots$};
\end{tikzpicture}
\caption{An excerpt of the $r$-homotopy from (8) of Proposition \ref{PropRips}.}
\label{FigOpenRips}
\end{figure}
\end{proof}

\begin{Remark}
 Property (8) of Proposition \ref{PropRips} is a kind of openness condition for homotopy classes of $r$-loops of size $k+1$ in the following sense. It states that for a given $r$-loop $L$ of size $k+1$, all $r$-loops of the same size, which are pointwise close enough to $L$, are in fact $r$-homotopic to $L$. A related idea was used in Proposition 5.12 of \cite{Cha2}. Along with converging $r$-loops of the same size and naturally obtained limiting $r$-loops in a compact geodesic space $X$, we see that a converging sequence of $r$-loops (of constant size) stabilizes after finitely many steps in its $r$-homotopy type. We will use this intuitive idea several times throughout the paper, especially in Lemma \ref{LemmaConvergence}.
 
On a related note,  Property (9) of Proposition \ref{PropRips} is a variant of Corollary 6.2 of \cite{Cha2}.
\end{Remark}

\section{The size of holes}
\label{SecSizeHoles}

A one-dimensional hole in a space is usually considered to be given by a non-contractible loop. An intuitive goal of persistence is to measure the size of these holes, which is the main aim of this section. To that end we should be able to measure the size of loops. We will  be interested in subgroups generated by loops of size bounded by $r$. We introduce three ways of doing that. The first one uses the length of a path.

\begin{Def}\label{DefL} Let $r>0$.
 Given a rectifiable loop $\alpha$ in $X$, its \textbf{length} is denoted by $\length(\alpha)$. For any abelian group $G$ the subgroup $\L(X,r,G)\leq H_1(X,G)$ is generated by loops (cycles) of length less than $r$.
 
For the case of the fundamental group we have to proceed using based loops. We therefore utilise the lasso construction (as for example in \cite{Dyd}).  An $l$-\textbf{lasso} is a based loop of the form $\alpha * \beta * \alpha^-$, where $\alpha$ is a path \new{of finite length} starting at $\bullet$, and $\beta$ is a loop of length $l$ based at the endpoint of $\alpha$. A lasso is \textbf{geodesic} if $\alpha$ is a geodesic and $\beta$ is a geodesic circle. $\L(X,r,\pi_1) \leq \pi_1(X,\bullet)$ is generated by all $l$-lassos with $l<r$. \new{Each $l$-lasso is of finite length.}

Define also $\L(X,fin,\pi_1)=\cup_{n\in \NN}\L(X,n,\pi_1)$, $\L(X,fin,G)=\cup_{n\in \NN}\L(X,n,G)$, $\L(X,\pi_1)=\cap_{r>0}\L(X,r,\pi_1)$, and $\L(X,G)=\cap_{r>0}\L(X,r,G)$. \new{Observe that each element of these groups has a representative, of finite length. On the other hand, if a based loop $\alpha$ in $X$ is of finite length, then $[\alpha] \in \L(X,fin,\pi_1)$. }
\end{Def}

While intuitive, Definition \ref{DefL} has some drawbacks. As mentioned before (for the case of $HE$), not all homotopy classes of loops in a geodesic space have a representative of finite length. Therefore we introduce the radius of the smallest enclosing sphere as another measure of the size of a loop, which is finite for all loops and has some historical background. 

\begin{Def}\label{DefUr}
 Let $r>0$.  For any abelian group $G$ the subgroup $\S(X,r,G)\leq H_1(X,G)$ is generated by all loops, whose smallest closed enclosing ball is of radius less than $r$.
 
An $U_d$-\textbf{lasso} is a based loop of the form $\alpha * \beta * \alpha^-$, where $\alpha$ is a path starting at $\bullet$ and $\beta$ is a loop in some $\cB(x,d)$ based at the endpoint of $\alpha$. \new{An $U_d$-lasso could be of infinite length.} $\S(X,r,\pi_1) \leq \pi_1(X, \bullet)$ is generated by all $U_d$-lassos with $d<r$. 
Group $\S(X,\pi_1)=\bigcap_{r>0}\S(X,r,\pi_1) \leq \pi_1(X,\bullet)$ will be referred to as the \textbf{Spanier} group of $X$ following \cite{F}. In a similar fashion, $\S(X,G)=\bigcap_{r>0}\S(X,r,G) \leq H_1(X,G)$ will be referred to as the $H_1(\_,G)$-\textbf{Spanier} group of $X$.
\end{Def}

The third measure of a size of a loop is its diameter. Formally we introduce it here for completeness. The results containing it will be provided in Section \ref{SectFundGroup} using the interplay of the first two measures of a size of a loop, which we develop in this section. 

\begin{Def}\label{DefDr}
 Let $r>0$. Given a loop $\alpha$ in $X$ recall its  diameter $\diam(\alpha)=\max_{x,y\in \alpha}d(x,y)$.  $\D(X,r,\pi_1) \leq$ is generated by all based lassos $\alpha*\beta*\alpha^-$ with $\diam(\beta)<r$. For any abelian group $G$ the subgroup $\D(X,r,G)\leq H_1(X,G)$ is generated by loops (cycles) of diameter less than $r$.
 \end{Def}

Note that all groups defined above are monotonic with respect to $r$, i.e., increasing $r$ increases any of the group above (or keeps it the same). 

\begin{Proposition}\label{PropIneq}
 For any geodesic space $X, r>0$ and Abelian group $G$, the following hold:
 \begin{itemize}
 	\item $\S(X,r,\pi_1) \geq \L(X,2r,\pi_1)$;
	\item $\S(X,\pi_1) \geq \L(X,\pi_1)$;
 	\item $\S(X,r,G) \geq \L(X,2r,G)$;
	\item $\S(X,G) \geq \L(X,G)$.
\end{itemize} 
\new{For any geodesic space $X$ and  $r>0$:
\begin{itemize}
\item groups $\S(X,r,\pi_1)$ and $ \S(X,\pi_1)$ are normal subgroups of $\pi_1(X)$;
\item  groups $\L(X,2r,\pi_1)$ and $\L(X,\pi_1)$ are normal subgroups of $\L(X,fin,\pi_1)$.
\end{itemize}}
 
\end{Proposition}

\begin{proof}
The first part is easy (as any loop of circumference $l$ is contained in a closed $l/2$ ball around any of its points) and left to the reader.
As for the second part it suffices to prove the statement for $\S(X,r,\pi_1)$ and $\L(X,2r,\pi_1)$. In both cases the proof goes along the same lines. Suppose $\gamma$ is a (based) loop \new{of finite length} in $X$ and $\alpha*\beta*\alpha^-$ is an $l$-lasso. Then $\gamma *(\alpha*\beta*\alpha^-)*\gamma^-\simeq (\gamma*\alpha)*\beta*(\gamma*\alpha)^-$ is also an $l$-lasso  thus the groups in question are closed under conjugation. \new{The proof for $U_l$-lassos is analogous, with $\gamma$ being any based loop in $X$.}
\end{proof}

Now that we have various sizes of holes set up, we turn our attention to their role in the interplay between the discrete and continuous setting, formalized by fillings and samples. For the rest of this section we will develop results connecting (discrete) $r$-homotopies to (continuous) homotopies modulo loops of 'small' size.

The following proposition states that, up to holes of length $2r$, the $r$-fillings are unique. We prove it directly as the proof is an instructive model for similar but more elaborate proofs to come, even though it follows from Proposition \ref{PropNull}.

\begin{Proposition}\label{PropDif}
Given $r>0$ and an $r$-loop $L$ in a geodesic space $X$, consider any two of its fillings $\alpha$ and $\beta$. Then $[\alpha]_G - [\beta]_{G}\in \L(X,2r,G)$. If loops are based then $[\alpha * \beta^-]_{\pi_1}\in \L(X,2r,\pi_1)$. 
\end{Proposition}

\begin{proof}
 It suffices to prove the second statement. Suppose $L$ is given by points $x_0,x_1,\ldots,x_k$ and suppose that for each $i\in \ZZ (\textrm{mod } k+1)$ we have two geodesics $\alpha_i$ and $\beta_i$ from $x_i$ to $x_{i+1}$ so that $\alpha_i$ determine $\alpha$ and $\beta_i$ determine $\beta$. Define paths $\gamma_i=\beta_0 * \beta_1 * \ldots \beta_{i-1}$ for all $i>0$. Then
 $$
 \alpha * \beta^- \simeq \big(\alpha_0 * \beta_0^-\big) * \big( \gamma_1*(\alpha_1 * \beta_1^-)* \gamma_1^- \big)* \ldots
$$
$$
\ldots* \big( \gamma_k*
(\alpha_k * \beta_k^-)*
\gamma_k^-\big)
 $$
 is a concatenation of $(k+1)$-many $2r$-lassos. A geometric interpretation of the concatenation for the case $|L|=4$ can be found on Figure \ref{FigHoles}. The homological version of the proof is easier as it does not require us to base the loops by  $\gamma_i$'s.
 \begin{figure}
\begin{tikzpicture}
\node [circle,  scale=.3, fill=black,draw, label=above:$x_{1}$] (a1) at (-2,2) {};
\node [circle,  scale=.3, fill=black,draw, label=below:$x_{0}$] (a0) at (-2,-1) {};
\node [circle,  scale=.3, fill=black,draw, label=below:$x_{3}$] (a3) at (2,-1) {};
\node [circle,  scale=.3, fill=black,draw, label=above:$x_{2}$] (a2) at (2,2) {};
\draw [dashed] (a0) to [out = 60, in = -60]node[right] {$\beta_0$} (a1);
\draw (a0) to [out = 120, in = -120]node [left]{$\alpha_0$} (a1);
\draw [dashed] (a1) to [out = -30, in = -150] node [below]{$\beta_1$} (a2);
\draw  (a1) to [out = 30, in = 150] node [above]{$\alpha_1$} (a2);
\draw [dashed] (a3) to [out = 60, in = -60]node[right] {$\beta_2$} (a2);
\draw (a3) to [out = 120, in = -120]node [left]{$\alpha_2$} (a2);
\draw  (a0) to [out = -30, in = -150] node [below]{$\alpha_3$} (a3);
\draw [dashed] (a0) to [out = 30, in = 150] node [above]{$\beta_3$} (a3);
\end{tikzpicture}
\caption{A sketch of the decomposition of two $r$-samples of size $4$ from Proposition \ref{PropDif}. The solid line (filling $\alpha$)  differs from the dashed line  (filling $\beta$) by four loops $\alpha_i * \beta_i ^-$ of length less than $2r$.}
\label{FigHoles}
\end{figure}
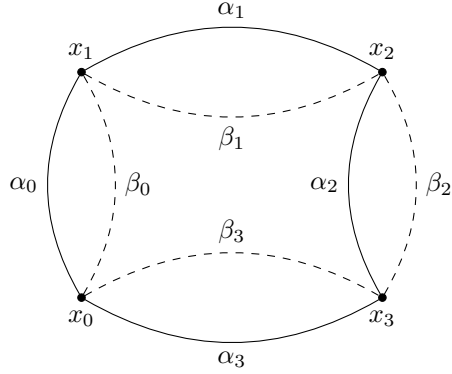
\end{proof}

\begin{Proposition}\label{PropKer}
Suppose $r>0$,  $\alpha$ is a loop  in a geodesic space $X$, $L$ is an $r$-sample of $\alpha$ and $\beta$ is any filling of $L$. Then $[\alpha]_G - [\beta]_{G}\in \S(X,r,G)$. If loops are based then $[\alpha * \beta^-]_{\pi_1}\in \S(X,r,\pi_1)$. 
\end{Proposition}

\begin{proof}
 The proof goes along the same lines as that of Proposition \ref{PropDif} and we will use the notation established there. Instead of being a geodesic, path $\alpha_i$ between $x_i=\alpha(t_i)$ and $x_{i+1}$ is now the restriction $\alpha|_{[t_i,t_{i+1}]}$ of possibly infinite length. By the definition of the $r$-sample we see that for each $i$ the loop $\alpha|_{[t_i,t_{i+1}]} * \beta_i^-$ is contained in $B(x_i,r)$. Hence we may use the same decomposition as suggested by Figure \ref{FigHoles} and the proof of Proposition \ref{PropDif} (producing a based decomposition using paths $\gamma_i$ in the case of $\pi_1$) to prove the claim.
\end{proof}

\begin{Cor}\label{CorStep}
 Suppose $X$ is a compact geodesic space. 
 
 If $X$ is SLSC then there exists $r>0$ so that for every loop $\alpha$, a filling of any $r$-sample of $\alpha$ is homotopic to $\alpha$ (the statement is true in based and in unbased settings).
 
If $G$ is an Abelian group and $X$ is $G$-SLSC, then there exists $r>0$ so that for every loop $\alpha$, a filling of any $r$-sample of $\alpha$ is homologous to $\alpha$ in $H_1(X,G)$.
\end{Cor}

\begin{proof}
If $X$ is compact and SLSC then there exists $r$ so that for each $x$ the inclusion $B(x,r)\to X$ induces a  map on $\pi_1$  that is trivial, and hence $\S(X,r,\pi_1)=0.$ The first statement follows by Proposition \ref{PropKer}. The second part may be proved in a similar fashion.
\end{proof}

The following two theorems are crucial for the arguments to come. Roughly speaking, the first one states that an $r$-loop is $r$-null if and only if its filling is induced by loops of length at most $3r$. As such it expresses $r$-nullhomotopies of $r$-loops in terms of groups $\L$. The second theorem expresses $r$-nullhomotopies of $r$-samples in terms of groups $\S$.

\begin{Proposition}\label{PropNull}
 Suppose $r>0$, $L$ is an $r$-loop  in a geodesic space $X$, and $\alpha$ is a filling of $L$. 
\begin{enumerate}
 \item If $L$ is based then the following holds:  $L$ is $r$-null if and only if $[\alpha]_{\pi_1} \in \L(X, 3r, \pi_1)$.
 \item For any Abelian group $G$ the following holds:  $0=[L]_G\in \Rips(X,r)$ if and only if  $[\alpha]_G\in \L(X,3r,G)$.
\end{enumerate}
  \end{Proposition}
  
\begin{proof}
We will prove (1) only. The proof of (2) is simpler as it does not require us to use based loops.

Suppose $L$, given by $\bullet=x_0, x_1, \ldots, x_k$, is $r$-null. An  $r$-nullhomotopy can be thought of as a maps on a triangulation $\Delta$ of a closed disc $D$ with vertex set $V$, edge set $E$ and the following properties:
\begin{itemize}
 \item the restriction to the boundary of $(D,\Delta)$ is $L$;
 \item given any pair of endpoints of an edge in $E$, their distance is less than $r$.
\end{itemize}
We can map the $1$-skeleton of $(D,\Delta)$ to $X$ in the obvious way: 
\begin{itemize}
 \item the map is identity on $V$ (recall that vertices of $\Rips(X,r)$ are points in $X$);
 \item given any pair of endpoints $x,y$ of an edge of $\Delta$, connect the points $x,y\in X$ by a geodesic (which is of length less than $r$);
 \item when connecting consecutive points of $L$ take the appropriate geodesic so that the induced filling on $L$ is $\alpha$.
\end{itemize}
 \begin{figure}
\begin{tikzpicture}
\coordinate (A) at (-1,0);
\coordinate (B) at (-5, 1);
\coordinate (C) at (-5, 3);
\coordinate (D) at (-1, 3);
\coordinate (A1) at (4,0);
\coordinate (B1) at (0, 1);
\coordinate (C1) at (0, 3);
\coordinate (D1) at (4, 3);
%
\draw [gray, top color=blue,bottom color=white, fill opacity =.3] (A) -- (B) -- (C) -- cycle;
\draw [gray, bottom color=blue,top color=white, fill opacity =.3] (A) -- (D) -- (C) -- cycle;
\draw [gray, top color=blue,bottom color=white,  opacity =.1, fill opacity =.1] (A1) -- (B1) -- (C1) -- cycle;
\draw [gray, bottom color=blue,top color=white,  opacity =.1, fill opacity =.1] (A1) -- (D1) -- (C1) -- cycle;
%
\draw[red, thick](A1) to  [bend left=25] node[below]{$r$}(B1);
\draw[red, thick](A1)to [bend left=15]node[left]{$r$}(C1);
\draw[red, thick](C1) to [bend left=15]node[left]{$r$} (B1);
\draw[red, thick](A1) to [bend left=25] node[right]{$r$}(D1);
\draw[red, thick](C1) to [bend left=25] node[above]{$r$}(D1);
\draw [->](2.5,2) arc (-30:240:.5);
\draw (2, 2.2) node {$<3r$};
\draw [fill] (A) circle [radius=0.1];
\draw [fill] (B) circle [radius=0.1];
\draw [fill] (C) circle [radius=0.1];
\draw [fill] (D) circle [radius=0.1];
\draw [fill] (A1) circle [radius=0.1];
\draw [fill] (B1) circle [radius=0.1];
\draw [fill] (C1) circle [radius=0.1];
\draw [fill] (D1) circle [radius=0.1];
\end{tikzpicture}
\caption{
A sketch of a map $\f$ of Proposition \ref{PropNull}. Given a configuration of abstract triangles on the left, construct the appropriate system of geodesics (red lines) of length less than $r$ (as suggested by the red label) on the right. Note that the decomposition into triangles on the left corresponds to the decomposition into loops of length less than $3r$ on the right.
}
\label{FigRips}
\end{figure}
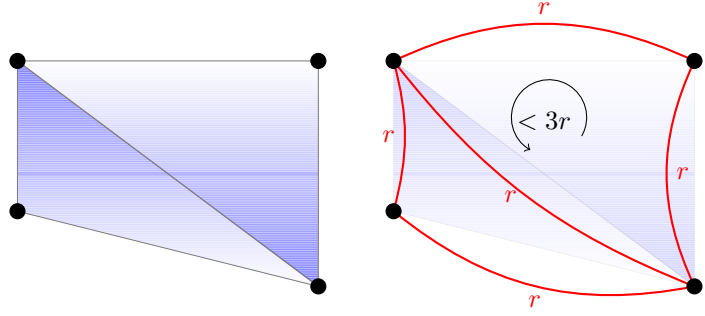
We call this map $\f$, see Figure \ref{FigRips}.
Loop $\alpha$ can now be decomposed as a concatenation of lassos generated by triangles of $\Delta$ in the following way. First orient all triangles of $\Delta$ consistently with $L$. For each triplet $T=[x,y,z]$ forming an oriented triangle in $\Delta$ consider a loop $\beta_T$, which is a filling of the $r$-loop $x,y,z$ induced by $\f$, i.e., for each pair of points, the chosen geodesic between them is the one chosen in the construction of $\f$. It is well known (and can be easily proved by induction) that we may also choose \new{finite length} paths $\gamma_T$ based at $\bullet$ so that $\alpha$ can be expressed as a concatenation of $l_T$-loops $\gamma_T *\beta_T * \gamma_T^-$ in the appropriate order, where $T$ goes through the set of all triangles $\Delta_2$ of $\Delta$ and $l_T< 3r, \forall T$. A simple example is provided by Figure \ref{FigDecomposition}. Thus $\alpha \in \L(X, 3r, \pi_1)$. In the homological case paths $\gamma_T$ are obsolete and $[\alpha]_G=\sum_{T\in \Delta_2} [\beta_T]_G$.

 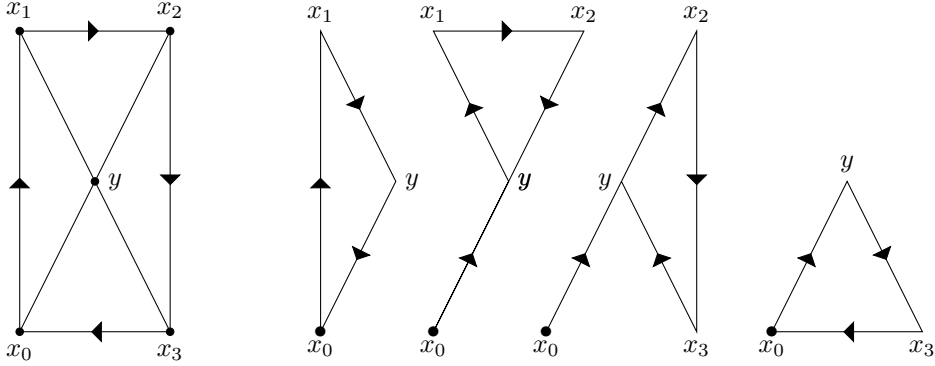
\begin{figure}
\begin{tikzpicture}
%

%
\node [circle,  scale=.3, fill=black,draw, label=above:$x_{1}$] (a1) at (-1,2) {};
\node [circle,  scale=.3, fill=black,draw, label=below:$x_{0}$] (a0) at (-1,-2) {};
\node [circle,  scale=.3, fill=black,draw, label=below:$x_{3}$] (a3) at (1,-2) {};
\node [circle,  scale=.3, fill=black,draw, label=above:$x_{2}$] (a2) at (1,2) {};
\node [circle,  scale=.3, fill=black,draw, label=right:$y$] (b) at (0,0) {};
\begin{scope}[ every node/.style={sloped,allow upside down}]
\draw (a0)-- node{\midarrow}(a1) -- node{\midarrow}(a2)--node{\midarrow}(a3)--node{\midarrow}(a0)--(b)--(a2);
\draw (a1)--(b)--(a3);
\draw (3,-2)node{$\bullet$}--(3,-2)node[below]{$x_0$}--node{\midarrow}(3,2)node[above]{$x_1$}--node{\midarrow}(4,0)node[right]{$y$}--node{\midarrow}(3,-2);
\draw (4.5,-2)node{$\bullet$}--(4.5,-2)node[below]{$x_0$}--node{\midarrow}(5.5,0)node[right]{$y$}--node{\midarrow}(4.5,2)node[above]{$x_1$}--node{\midarrow}(6.5,2)node[above]{$x_2$}--node{\midarrow}(5.5,0)node[right]{$y$}--(4.5,-2);
\draw (6,-2)node{$\bullet$}--(6,-2)node[below]{$x_0$}--node{\midarrow}(7,0)node[left]{$y$}--node{\midarrow}(8,2)node[above]{$x_2$}--node{\midarrow}(8,-2)node[below]{$x_3$}--node{\midarrow}(7,0);
\draw (9,-2)node{$\bullet$}--(9,-2)node[below]{$x_0$}--node{\midarrow}(10,0)node[above]{$y$}--node{\midarrow}(11,-2)node[below]{$x_3$}--node{\midarrow}(9,-2);
\end{scope}
\end{tikzpicture}
\caption{An excerpt from the proof of Proposition \ref{PropNull}. Given an $r$-null $r$-loop $x_0,x_1,x_2,x_3$ on the left, an $r$-nullhomotopy is given by a triangulation $\Delta$ of the disc (depicted as the rectangle on the left in this figure). In this case $\Delta$ contains an additional vertex $y$ and four triangles in $\Rips(X,r)$. Thinking of vertices of $\Delta$ as points in $X$, we may replace the edges of $\Delta$ by geodesics and obtain the same scheme in $X$ with $\alpha$ along the boundary. Loop $\alpha$ is decomposed into four loops of length less than $3r$ and the ordered loops on the right are the four lassos, whose concatenation in the suggested order demonstrates that $\alpha \in  \L(X, 3r, \pi_1)$.}
\label{FigDecomposition}
\end{figure}

Now suppose $\alpha \in \L(X, 3r, \pi_1)$. By Proposition \ref{PropRips} (6) we may assume $\alpha$ to be a concatenation of $3r$-lassos. Hence it suffices to prove that any lasso of length less than $3r$ is $r$-null. Assume therefore $\alpha = \gamma * \beta *\gamma^-$ where $\gamma$ is a based path and $\beta$ is a loop of length less than $3r$. It suffices to prove that any $r$-sample of $\beta$ is $r$-null (in the unbased sense), which is true by Proposition \ref{PropRips}(5).
\end{proof}

\begin{Theorem}\label{ThmNull}
 Suppose $r>0$, $\alpha$ is a loop  in a geodesic space $X$, and $L$ is an $r$-sample of $\alpha$. 
\begin{enumerate}
 \item If $\alpha$ is based then the following holds:  $L$ is $r$-null if and only if $[\alpha]_{\pi_1} \in \S(X, 3r/2, \pi_1)$.
 \item For any Abelian group $G$ the following holds:  $0=[L]_G\in \Rips(X,r)$ if and only if  $[\alpha]_G\in \S(X,3r/2,G)$.
\end{enumerate}
  \end{Theorem}

\begin{proof}
 We will focus on (1), the proof of (2) is again simpler. 
 
 Suppose $L$ is $r$-null. Choose a filling $\beta$ of $L$. Observe that:
 \begin{itemize}
 	\item $[\beta]_{\pi_1} \in \L(X, 3r, \pi_1)$ by Proposition \ref{PropNull};
	\item $[\alpha * \beta^-]_{\pi_1}\in \S(X, r, \pi_1)$ by Proposition \ref{PropKer}.
\end{itemize}
By Proposition \ref{PropIneq} we have $[\alpha]_{\pi_1} \in \S(X, 3r/2, \pi_1)$.

Now suppose $[\alpha]_{\pi_1} \in \S(X, 3r/2, \pi_1)$. By Proposition \ref{PropRips} (6) it suffices to prove that any $U_d$-lasso with $d<3r/2$ is $r$-null. Assume therefore $\alpha = \gamma * \beta *\gamma^-$ where $\gamma$ is a based path and $\beta$ is an $U_d$-lasso with $d<3r/2$. It suffices to prove that any $r$-sample of $\beta$ is $r$-null (in the unbased sense). 

Choose $w\in X$ so that $\beta \subset \cB(w, d)$. Additionally assume $w\in \beta$: if that is not the case connect some point on $\beta$ to $w$ by a geodesic and then backtrack. Inserting such detour to $\beta$ does not change the situation (containment on $\cB(w, d)$ or the homotopy type of $\beta$,) but it does add $w$ to $\beta$. Define step $s=3r-2d$. By Proposition \ref{PropRips}(4) we may assume $L$ is an $s$-loop determined by $w=x_0, x_1, \ldots, x_k, x_{k+1}=w\in X$. 
For each $i$ define $y_i$ as some midpoint between $x_i$ and $w$  as in Figure \ref{FigSpanier}.
Note that $L$ is $r$-homotopic to the join $L_0 *L_1*\ldots*L_k$, where each $r$-loop $L_i$ is given by the sequence  $w, y_i, x_i, x_{i+1},y_{i+1},w$. Since points $y_i$ are midpoints, the length of each $L_i$ is of length $d(w, x_i)+d(x_i, x_{i+1}) + d(x_{i+1},)< d + s + d =3r$ by the definition of $s$. Thus by Proposition \ref{PropRips}, each $L_i$ is $r$-null and so is $L$. 
\end{proof}

 \begin{figure}
\begin{tikzpicture}
\draw [thick](0,0) circle (2);
\draw (0,0) node{$\bullet$} node[above]{$w$};
\draw (2,0) node[right]{$\cB(w, d)$};
\draw (0,1.4) node{$\beta$};
\draw plot [smooth ] coordinates {(0,0)(.8,.2)(0,1.2)(-1.9,0)(-.4, -1.4)(1.5,-.6)(0,0)};
\draw (-1.5,.4) node{$\bullet$} node[above]{$x_i$};
\draw (-1.5,-0.7) node{$\bullet$} node[right]{$x_{i+1}$};
\draw[dashed] (0,0) -- (-1.5, .4);
\draw[dashed] (0,0) -- (-1.5, -.7);
\draw (-.75, .2) node {$\bullet$} node[above] {$y_i$};
\draw (-.75, -.35) node {$\bullet$} node[below right] {$y_{i+1}$};
\end{tikzpicture}
\caption{A sketch from the proof of Theorem \ref{ThmNull}, representing an $r$-loop $L_i$ given by the sequence  $w, y_i, x_i, x_{i+1},y_{i+1},w$. The step parameter $s$ is chosen so that the length of such $L_i$ is less than $3r$.}
\label{FigSpanier}
\end{figure}
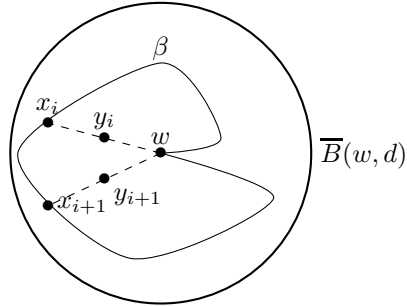

\section{The correspondence theorems}
\label{SectFundGroup}

The results of Section \ref{SecSizeHoles} relate $r$-homotopies in Rips complexes and homotopies in the underlying space. In this section, we extend this to a relation between the fundamental group (and the first homology groups) of the Rips complex, and modified fundamental group (and the first homology groups) of the underlying space. 
We prove various correspondence theorems: Theorems \ref{ThmPerSpa}, \ref{ThmPerCirc}, and \ref{ThmPerDiam}. These make use of groups $\L, \S, $ and $\D$ respectively, to precisely describe how persistence measures one-dimensional holes. As usual, the discrete-continuous interplay of the correspondence is induced by samples and fillings.

We conclude the section with a technical result stating that fundamental and first homology groups of Rips complexes of compact spaces are often finitely generated. 

\subsection{Radius of the enclosing ball as a measurement of a hole}

\begin{Definition}\label{DefRho}
Suppose $r>0$, $X$ is a geodesic space and $G$ is an Abelian group. Maps
$$
\rho_r^{\pi_1} \colon \pi_1(X, \bullet) \to \pi_1(\Rips(X,r),\bullet),
$$
$$
\rho_r^{G} \colon H_1(X,G) \to H_1(\Rips(X,r),G)
$$
are defined by mapping a representative $[\alpha]_*$ to any $r$-sample.
Maps $\rho_r^{\pi_1}$ and $\rho_r^G$ are well defined by Proposition \ref{PropRips} (4), (6), and are obviously a homomorphism. 
\end{Definition}

\begin{Proposition}\label{PropPrePers}
Let $X$ be a geodesic space. Then the following hold for all $r>0$:
\begin{enumerate}
 \item maps $\rho_r^{\pi_1}$  are surjective;
 \item the kernel of $\rho_r^{\pi_1}$ is $\S(X, 3r/2, \pi_1)$;
 \item maps $\rho_r^{\pi_1}$ commute with inclusion induced maps $$
 \pi_1(i_{p,q})\colon \pi_1(\Rips(X,r),\bullet)\to \pi_1(\Rips(X,q),\bullet);
 $$
 \item suppose $X$ is compact. Then $X$ is SLSC if and only if $\rho_r^{\pi_1}$ is an isomorphism for small  values of $r$.
\end{enumerate}
\end{Proposition}

\begin{proof}
(1) follows from the fact that each $r$-loop admits a filling. (2) is  Theorem \ref{ThmNull}(1). (3) follows directly from the definitions of maps or by Proposition \ref{PropRips} (2). 

It remains to prove (4). Compact $X$ is SLSC iff there exists $t>0$ so that for each $x\in X$ the inclusion $B(x,t)\to X$ induces the trivial map on the fundamental group. This in turn is equivalent to $\S(X,t,\pi_1)=0$ which by (1) and (2) is equivalent to the fact that $\rho_t^{\pi_1}$ is an isomorphism. The later also implies that $\rho_{t'}^{\pi_1}$ is an isomorphism for all $t'<t$.
\end{proof}

\begin{Proposition}[A homological version of Proposition \ref{PropPrePers}] \label{PropHPrePers}
Let $X$ be a geodesic space and $G$ an Abelian group. Then the following hold for all $r>0$:
\begin{enumerate}
 \item maps $\rho_r^{G}$  are surjective;
 \item the kernel of $\rho_r^{G}$ is $\S(X, 3r/2, G)$;
 \item maps $\rho_r^{G}$ commute with inclusion induced maps $$
 H_1(i_{p,q},G)\colon H_1(\Rips(X,r), G)\to H_1(\Rips(X,q), G);
 $$
 \item suppose $X$ is compact. Then $X$ is $G$-SLSC if and only if $\rho_r^{G}$ is an isomorphism for small values of $r$.
\end{enumerate}
\end{Proposition}

\begin{proof}
The proof uses Theorem \ref{ThmNull}(2) and is practically the same as that of Proposition \ref{PropPrePers}.
\end{proof}

\begin{Theorem}\label{ThmPerSpa}
 [Persistence-Spanier correspondence Theorem]
 Let $X$ be a geodesic space and $G$ an Abelian group.
  Maps $\rho^{\pi_1}_r$ provide an isomorphism  $$\{\pi_1(\Rips(X,r)\bullet)\}_{r>0}\cong\{\pi_1(X, \bullet)/\S(X, 3r/2, \pi_1)\}_{r>0}.$$
  Maps $\rho^{G}_r$ provide an isomorphism $$\{H_1(\Rips(X,r), G)\}_{r>0}\cong\{H_1(X, \bullet)/\S(X, 3r/2, G)\}_{r>0}.$$
\end{Theorem}

\begin{proof}
 Follows from Propositions \ref{PropPrePers} and \ref{PropHPrePers}.
\end{proof}

\subsection{The length of a representative as a measurement of a hole}

\begin{Def}\label{DefLambda}
Suppose $r>0$, $X$ is a geodesic space and $G$ is an Abelian group. Maps
 $$
 \lambda_r^{\pi_1}\colon \pi_1(\Rips(X,r),\bullet)\to \L(X,fin, \pi_1)/
\L(X, 3r, \pi_1),
 $$ 
  $$
 \lambda_r^{G}\colon H(\Rips(X,r),G)\to \L(X,fin, G)/
\L(X, 3r, G)
 $$
 are defined by mapping an $r$-loop to its filling.
\end{Def}

\begin{Proposition}\label{PropLambda}
 For each $r>0$, each Abelian group $G$ and each geodesic space $X$, maps $ \lambda_r^{\pi_1}$ and $ \lambda_r^{G}$are well defined isomorphisms. Furthermore, maps $\lambda^*_r$ commute with the inclusions $i_{p,q}\colon \Rips(X,p)\to \Rips(X, q)$ induced maps and the obvious quotient maps $\L(X,fin, *)/\L(X, 3p, *)\to \L(X,fin, *)/\L(X, 3q, *)$ for all $p<q$.
\end{Proposition}

\begin{proof}
Maps are well defined by Propositions \ref {PropDif} and \ref{PropNull}. They are injective by Proposition \ref{PropNull}. It remains to prove they are also surjective. We provide the proof for $ \lambda_r^{\pi_1}$, the homological case can be proved in an identical manner. 
 
 Take any (based) loop $\alpha$ in $X$ of finite length. Choose an $r$-sample $L$ of $\alpha$ so that the length of $\alpha$ between each pair of consecutive points is less than $r$. It is easy to see (check for example a similar argument in Proposition \ref{PropDif}) that for each filling $\beta$ of $L$ we have $[\alpha * \beta^-]_{\pi_1}\in \L(X,2r,\pi_1)$, hence $L$ is mapped to the equivalence class of $\alpha$ in $ \L(X,fin, \pi_1)/\L(X, 3r, \pi_1)$ by $\lambda_r^{\pi_1}$.
 
The final claim about commutativity follows easily from the definitions, as in both direction we map a class represented by an $r$-loop $L$ to a class represented by a filling of $L$.
 \end{proof}

\begin{Theorem}\label{ThmPerCirc}
 [Persistence-circumference correspondence Theorem]
 Let $X$ be a geodesic space and $G$ an Abelian group.
  Maps $\lambda^{\pi_1}_r$ provide an isomorphism  $$\big\{\pi_1(\Rips(X,r)\bullet)\big\}_{r>0}\cong\big\{\L(X,fin, \pi_1)/
\L(X, 3r, \pi_1)
\big\}_{r>0}.$$
  Maps $\lambda^{G}_r$ provide an isomorphism $$\big\{H_1(\Rips(X,r), G\big)\}_{r>0}\cong\big\{\L(X,fin,G)/
\L(X, 3r, G)
\big\}_{r>0}.$$
Furthermore, if $X$ is SLSC then 
$$\big\{\pi_1(\Rips(X,r)\bullet)\big\}_{r>0}\cong\big\{\pi_1(X, \bullet)/
\L(X, 3r, \pi_1)
\big\}_{r>0}.$$
Also, if $X$ is $G$-SLSC then 
$$\big\{H_1(\Rips(X,r), G)\big\}_{r>0}\cong\big\{H(X,G)/
\L(X, 3r, G)
\big\}_{r>0}.$$
\end{Theorem}

\begin{proof}
 Follows from Propositions \ref{PropLambda} and (for the SLSC case) \ref{PropLength}. 
\end{proof}

\subsection{The diameter of a representative  as a measurement of a hole}

The third way to measure the size of a hole is to consider the diameter, as described in Definition \ref{DefDr}. We include the corresponding result here for completeness but use a result of a forthcoming section in the proof.

Propositions  \ref{PropIneq} and \ref{PropDiam} describe a general relationship between groups $\L, \S,$ and $\D$.

\begin{Proposition}
 \label{PropDiam}
 Suppose $X$ is a compact geodesic space, $G$ is an Abelian group, and $r>0$. Then $\S(X, r, \pi_1)=\D(X, r, \pi_1)$ and $\S(X, r, G)=\D(X, r, G)$.
\end{Proposition}

\begin{proof}
As usually, we will only provide a proof for the $\pi_1$ case. The inclusion $S(X, r, \pi_1)\supset\D(X, r, \pi_1)$ is obvious as any set of diameter $d$ is contained in the closed ball of radius $d$ around any of its points.
In remains to prove $S(X, r, \pi_1)\subset\D(X, r, \pi_1)$.

By Theorem \ref{PropFin} \textbf{a} and  an inductive use of Theorem \ref{PropFin} \textbf{d}, $\S(X, r, \pi_1)$ is generated by $\S(X, r/2, \pi_1)$ and  a collection of geodesic $l_i$-lassos $\alpha_i$ with $r \leq l_i<2r$. 
It is easy to see that $\S(X, r/2, \pi_1) \subset \D(X, r, \pi_1)$ since every cover by balls of radius $r/2$ has diameter at most $r$. Furthermore, each loop of length $\ell$ is of diameter at most $\ell/2$  hence $[\alpha_i]_{\pi_1}\in \D(X, r, \pi_1), \forall i$. We conclude that $\S(X, r, \pi_1)\subset\D(X, r, \pi_1)$.
\end{proof}
 
 \begin{Theorem}\label{ThmPerDiam}
 [Persistence-diameter correspondence Theorem]
 Let $X$ be a geodesic space and $G$ an Abelian group.
Then  $$\{\pi_1(\Rips(X,r)\bullet)\}_{r>0}\cong\{\pi_1(X, \bullet)/\D(X, 3r/2, \pi_1)\}_{r>0},$$
 $$\{H_1(\Rips(X,r), G)\}_{r>0}\cong\{H_1(X, G)/\D(X, 3r/2, G)\}_{r>0}.$$
\end{Theorem}

\begin{proof}
Follows from Proposition \ref{PropDiam} and Theorem \ref{ThmPerSpa}.
\end{proof}

\subsection{Fundamental groups of Rips complexes are finitely generated}

\begin{Def}
 Suppose $r>0$, $G$ is an Abelian group, and $X$ is a geodesic space. Let $S\subset X$ be an $(r/3)$-\textbf{dense sample} of $X$, i.e., $\forall x\in X \ \exists s\in S: d(x,s)<r/3$. Suppose $\bullet \in S$. Homomorphisms 
  $$
 \mu_r^{\pi_1}\colon \pi_1(\Rips(S,r),\bullet)\to \pi_1(\Rips(X,r),\bullet),
 $$ 
   $$
 \mu_r^{G}\colon H_1(\Rips(S,r),G)\to H_1(\Rips(X,r),G)
 $$
 are induced by the inclusion $S \to X$, with the metric on $S$ being the restriction of the metric on $X$.
\end{Def}

\begin{Remark}
 We will frequently assume that $\bullet$ is contained in a $*$-dense sample $S$. The reason is that this allows a simpler formulation of results related to the fundamental group. Mathematically this assumption could be circumvented and does not restrict our results.  For more results on map $\mu$ see \cite{ZV1}.
 \end{Remark}

\begin{Proposition} \label{PropFG}
Suppose $r>0$, $G$ is an Abelian group,   $X$ is  geodesic, and $S_r\subset X$ is an $r/3$-dense sample of $X$. Then maps $ \mu_r^{\pi_1}$ and $ \mu_r^{G}$ are  epimorphisms.
\end{Proposition}

\begin{proof}
We only provide  a proof for $ \mu_r^{\pi_1}$.
We assume $\bullet \in S$. Suppose an $r$-loop $L$ representing an element in $\pi_1(\Rips(X,r),\bullet)$ is given. Applying Proposition \ref{PropRips}(8) to a filling of $L$,  we may assume $L$ is actually an $(r/3)$-loop given by $\bullet=x_0, x_1, \ldots, x_k, x_{k+1}=\bullet$. For each $i$ choose $y_i$ so that $d(x_i, y_i)< r/3, \forall i$, and $y_0=y_{k+1}=\bullet$. Then $d(y_i, y_{i+1})<r$ and $d(y_i, x_{i+1})<2r/3$ for all $i$, so  $\bullet=y_0, y_1, \ldots, y_k, y_{k+1}=\bullet$  is an $r$-loop in $\Rips(S,r)$, which is homotopic to $L$ by the $r$-homotopy depicted on Figure \ref{FigOpenRips}.
\end{proof}

Recall that a metric space $X$ is totally bounded, if for each $\eps>0$ space $X$ can be covered by a finite collection of balls of radius $\eps$.

\begin{Corollary}\label{CorFG}
Let $r>0$ and let $G$ be an Abelian group. If a geodesic space $X$ admits a finite $r/3$-dense sample then $\pi_1(Rips(X,r), \bullet)$ and $H_1(Rips(X,r), G)$ are finitely generated. In particular, if $X$ is a totally bounded space (a class which includes compact metric spaces) then for each $r>0$, groups $\pi_1(\Rips(X,r),\bullet)$ and  $H_1(Rips(X,r), G)$ are finitely generated. 
\end{Corollary}

\begin{proof}
Again we only provide the proof for $\pi_1$.
 If an $r/3$-dense sample $S$ is finite then $\pi_1(\Rips(S,r),\bullet)$ is a fundamental group of a finite complex hence finitely generated. By Proposition \ref{PropFG} so is its image $\pi_1(\Rips(X,r),\bullet)$. The last statement follows from Propositions \ref{PropPrePers} (4) and \ref{PropLambda}. 
\end{proof}

\begin{Remark}
\label{RemTame}
 Combining Corollary \ref{CorFG} and Proposition \ref{PropRips} (9) we see that for a field $\FF$ and a compact geodesic space $X$ the $H_1(\_, \FF)$-persistence of $X$ via open Rips complexes is q-tame (see \cite{Cha1} for a definition) hence the corresponding persistence diagrams are well defined. It follows from later results of this paper that the same conclusion holds for  open Rips complexes and open or closed \v Cech complexes.
 \end{Remark}

\begin{Remark}
 Combining Corollary \ref{CorFG} and Proposition \ref{PropPrePers} (4) we obtain a new proof that the fundamental groups of a compact geodesic space is finitely generated.
\end{Remark}

\section{Loops in the limit}
\label{SectLimit}

In this section we present two ways of obtaining a loop as a limit of a sequence of loops. In Proposition \ref{PropLength} we consider SLSC spaces and prove that the limiting argument may be used to obtain representatives of minimal length. In Proposition \ref{PropHLength} we provide a homological version. 
Technical Lemma \ref{LemmaConvergence} is the second convergence result of this section. In the forthcoming section it plays a crucial role in identifying geodesic loops, which correspond to critical values of persistence. It shows that, up to short loops, a sequence of 'converging' geodesics stabilizes after finitely many steps. 

\begin{Prop}\label{PropLength}
If $X$ is a SLSC geodesic space then each element of $\pi_1(X, \bullet)$ has a representative of finite length. Furthermore, if $X$ is compact then:
\begin{enumerate}
 \item each element of the fundamental group has a representative of minimal length and therefore $\L(X, fin, \pi_1)=\pi_1(X, \bullet)$;
  \item each free homotopy class of a loop has a representative of minimal length;
 \item each element of the fundamental group has a representative in the form of an $l$-lasso $\alpha=\gamma * \beta * \gamma^-$ of minimal length $l$;
 \item $\S(X,r,\pi_1)=\L(X,2r,\pi_1)$.
\end{enumerate}
\end{Prop}

\begin{proof}
 By Proposition \ref{PropPrePers} there exists $t>0$ so that $\rho_t^{\pi_1}$ is an isomorphism. By Proposition \ref{PropKer},  each element of the fundamental group has a representative which is a filling of some $t$-loop in $X$, hence is of finite length.
 
 (1): We will obtain the required representative as the limit of representatives $\alpha_n$ defined below. Suppose $\alpha$ is a non-contractible loop. Define 
 $\ell=\inf_{\beta \simeq \alpha} \length(\beta).$ By the previous paragraph $\ell<\infty$ and since $X$ is SLSC, $\ell>0$. Choose also $t>0$ so that for each $x\in X$ the inclusion $B(x,3t/2)\to X$ induces the trivial map on the fundamental group. By Proposition \ref{PropPrePers} applied to our setting, any filling of a $t$-loop is homotopically unique. For each $n\in \NN$ choose $\alpha_n \simeq \alpha, \length(\alpha_n)< \ell + 1/n$. For each $n\in \NN$ we now choose any $(t/2)$-samples $L_n$ of $\alpha_n$ given by 
 $$
 \bullet = x_0^n, x_1^n, \ldots, x_k^n, x_{k+1}^n=\bullet,
 $$ 
where $k=\size(L_n)$ is fixed (we may choose $k$ as any integer greater than $2(\ell+1)/t$). By compactness there is a subsequence of $(L_n)_n$ so that for each $i$ the sequence $(x_i^n)_n$ converges to some $x_i$ in $X$. Consider now a $t$-loop $L$ given by $\bullet = x_0, x_1, \ldots, x_k, x_{k+1}=\bullet$ and its filling $\beta$, both of length $\ell$. By Proposition \ref{PropRips}(8) $L$ is $t$-homotopic to $L_N$ for some large enough $N$. By Propositions \ref{PropNull} and \ref{PropIneq} $[\beta * \alpha_N^-]_{\pi_1}\in \S(X,3t/2,\pi_1)=0$, hence $\beta$ is a representative of $[\alpha]_{\pi_1}$ of length $\ell$. 

The proof of (2) is  the same, considering unbased loops instead of based ones.

(3) Given $[\alpha]_{\pi_1}$ we use (2) to choose a minimal free homotopy representative $\beta$. It is  known that there exists path $\gamma$ from $\bullet$ to the corresponding point in $\beta$ so that the lasso $\gamma * \beta * \gamma^-$ is (based) homotopic to $\alpha$.  If this lasso representative was not minimal in the required sense then $\beta$ wouldn't be a shortest representative of $[\alpha]_{\pi_1}$. 

(4) By Proposition \ref{PropIneq} we only need to show $\S(X,r,\pi_1)\subseteq \L(X,2r,\pi_1)$. Take any lasso $\alpha=\gamma * \beta * \gamma^-$, where $\beta$ is a loop in some $B(x,r)$. \new{As $X$ is SLSC, we can assume $\gamma$ is of finite length. (The proof of this fact is similar to a corresponding proof for the loops: path $\gamma$ is homotopic rel boundary to any filling of its $t$-sample.).} Define $D=\max_{w\in \beta}d(w,x), \eps = r-D$. By Corollary \ref{CorStep} there exists $T<\eps$ so that for any chosen $t<T, \beta$ is homotopic to a filling of its  $t$-sample $L_\beta$. Suppose $L_\beta$ is given by points $x_0=x, x_1, \ldots, x_k, x_{k+1}=x$. For   
each $i$ choose a geodesic $\delta_i$ from $x$ to $x_i$ and a geodesic $\eta_i$ from $x_i$ to $x_{i+1}$. Now for each $i$ define a $(2D+\eps)$-loop $\beta_i = \delta_i * \eta_i * \delta_{i+1}^-$ and notice that $2D+\eps<2r$. Then $\alpha$ is homotopic to the concatenation of $(2D+\eps)$-loops 
$\alpha \simeq \beta_0 * \beta_1* \ldots * \beta_k$ hence $\alpha \in \L(X,2r,\pi_1)$.
\end{proof}

\begin{Remark}
Statements (1)-(3) of Proposition \ref{PropLength} are not valid if $X$ is not compact. A counterexample for (1) is $S^1 \times (0,1] \cup \{(x_0,0)\}\subset \RR^3$, where $x_0$ is some point on $S^1$ and $\bullet =(x_0,0)$. A counterexample for (2)  and (3) is the surface of revolution obtained by rotating the graph of $f(x)=1/x+1$ for $x>0$ around $x$-axis. 
\end{Remark}

\begin{Prop}\label{PropHLength}
If $G$ is an Abelian group and $X$ is a $G$-SLSC geodesic space then each element of $H_1(X,G)$ has a representative of finite length. Furthermore, if $X$ is compact then:
\begin{enumerate}
 \item each element of $H_1(X,G)$ has a representative of minimal length and therefore $\L(X, fin, G)=H_1(X, G)$;
 \item $\S(X,r,G)=\L(X,2r,G)$.
\end{enumerate}
\end{Prop}

\begin{proof}
 The proof is similar as that of  the equivalent statements in Corollary \ref{PropLength}.
\end{proof}

\begin{Def}\label{DefGalpha}
 Suppose $\alpha$ is an oriented simple closed curve \new{of finite length} in $X$. Define $G_\alpha \new{\subseteq \L(X, fin,  \pi_1)}$ as the collection of homotopy classes of all $\alpha$-lassos, i.e., loops of the form $\gamma * \alpha * \gamma^-$, where $\gamma$ is a path \new{of finite length} from $\bullet$ to any point on $\alpha$. Note that the concatenation takes place at the endpoint of $\gamma$.
\end{Def}

\begin{Remark}
 Similarly as in Proposition \ref{PropIneq} we may observe that  group $\langle G_\alpha \rangle$ generated by $G_\alpha$ is a normal subgroup of \new{$\L(X, fin,  \pi_1)$}. In fact, it is the normal closure of any single $\alpha$-lasso.
\end{Remark}

Suppose $r>0$. Given a sequence of loops of converging length in a compact geodesic space $X$, there may be no subsequence of loops, whose homotopy classes eventually stabilize. However, Lemma \ref{LemmaConvergence} explains that there exists a subsequence, whose homotopy types stabilize modulo short loops $ \L(X, 3r, \pi_1)$. In particular, this means that the homotopy type of the corresponding loops in the appropriate Rips complex stabilize. This lemma will be used extensively in Section \ref{SectStruc} to obtain limiting geodesic circles, arising from approximation of the critical values of a persistence via Proposition \ref{PropLambda}. In fact, these limiting geodesic circles will be the geometric features generating critical values of a persistence.

\begin{Lemma}\label{LemmaConvergence}
Suppose $X$ is a compact geodesic space and $r>0$. For each $i\in \NN$ let $\alpha_i$ be a  loop in $X$ of length $l_i$ and let $l_i\to l>0$ as $i \to \infty$. Then there exists a loop $\alpha$ of length at most $l$ in $X$ and an infinite set $J\subset \NN$ so that 
$$
G_{\alpha_j} \ \L(X, 3r, \pi_1) =G_{\alpha} \ \L(X, 3r, \pi_1), \quad \forall j\in J.
$$
If all $\alpha_i$ are geodesic circles then  the length of $\alpha$ equals $l$.

In particular, for each $j \in J$ and each path $\gamma$ \new{of finite length} from $\bullet$ to (some point on) $\alpha$ there exists a path $\gamma_j$  \new{of finite length} from $\bullet$ to $\alpha_j$ so that $\rho^{\pi_1}_r([\gamma_j *\alpha_j* \gamma_j ^-]_{\pi_1})=\rho^{\pi_1}_r([\gamma*\alpha*\gamma^-]_{\pi_1})$. Similarly,  for each $j \in J$ and each  path $\gamma_j$  \new{of finite length} from $\bullet$ to $\alpha_j$ there exists a path $\gamma$  \new{of finite length} from $\bullet$ to $\alpha$ so that $\rho^{\pi_1}_r([\gamma_j *\alpha_j* \gamma_j ^-]_{\pi_1})=\rho^{\pi_1}_r([\gamma*\alpha*\gamma^-]_{\pi_1})$.
\end{Lemma}

\begin{proof}
Define $k=\lceil 2 \sup_i \{l_i\} /r \rceil +1$. For each $i$ choose an $(r/2)$-sample $L_i$ of $\alpha_i$ having size $k$, given by points $x^i_0, x^i_1, \ldots, x^i_k, x^i_{k+1}=x^i_0$. Note  that for each $m$ the length of the path segment of $\alpha_i$ between consecutive points $x^i_m$ and $x^i_{m+1}$ is less than $r/2$. By taking an appropriate subsequence of $(\alpha_i)_i$ and considering compactness we may assume convergence $\lim_{i\to\infty}x^i_j=x_j\in X$ for all $j$. Let $L$ be an $r$-loop determined by $x_0, x_1, \ldots, x_k, x_{k+1}=x_0$ and choose its filling $\alpha$. 
$$
l=\length(\alpha)=\sum_{j=1}^k d(x_j, x_{j+1})=\lim_{i \to \infty}\sum_{j=1}^k d(x^i_j, x^i_{j+1})\leq\lim_{i \to \infty} l_i.
$$
(If $\alpha_i$ are geodesic circles then the inequality in the above line is the equality and the length of $\alpha$ is $l$.)
Furthermore, by Proposition \ref{PropRips} (8) there exists $N\in 
\NN$ so that $L_i$ is $r$-homotopic to $L$ for all $i>N$ (meaning, in particular, that $x_j^i$ and $x^i$ are appropriately close). We will prove that for $i>N$ we have $G_{\alpha_i} \L(X, 3r, \pi_1) =G_{\alpha} \L(X, 3r, \pi_1)$, so the set of indices $i>N$, pulled back to the original sequence, corresponds to the set $J$.

Fix $i>N$. Choose a path $\gamma$ \new{of finite length} from $\bullet$ to (some point on) $\alpha$. We prolong $\gamma$ along (any direction of) $\alpha$ to $x_0$. Note that this does not change the homotopy type of $\gamma * \alpha * \gamma^-$. Define also $\gamma_i$ as a prolongation of $\gamma$ by a geodesic $\mu_i$ from $x_0$ to $x^i_0$. Define for each $i$ define $\widetilde L_i$ as $x_0, x^i_0, x^i_1, \ldots, x^i_k, x^i_{k+1}=x^i_0, x_0$, i.e., $\widetilde L_i$ is the $r$-loops obtained by appending and prepending $x_0$ to $L_i$. Note that $\widetilde L_i$ is an $r$-sample of $\mu_i * \alpha_i * \mu_i^-$.
We now combine the following arguments:
\begin{itemize}
 	\item similarly as in Proposition \ref{PropDif} we observe that for any $i$ and for any filling $\tilde \alpha_i$ of $L_i$ we have $[\gamma_i * \alpha_i * \tilde \alpha_i ^- * \gamma_i^-]_{\pi_1} =[\gamma_i * \alpha_i * \gamma_i^-]_{\pi_1} \cdot [\gamma_i * \tilde\alpha_i^- * \gamma_i^-]_{\pi_1}\in \L(X, 2r, \pi_1)$ due to the construction in the first paragraph in this proof;
 	\item the approach of Propositions \ref{PropRips} (8) (convergence and $r$-homotopic equivalence of the appropriate $r$-samples in terms of relation of the fillings) applied to loops $\alpha $ and $\mu_i * \tilde \alpha_i * \mu_i^- $ implies that $L$ and $\widetilde{L}_i$ are (based) $r$-homotopic;
 	\item Proposition \ref{PropNull} (expressing $r$-homotopic relation in terms of $\L(X, 3r, \pi_1)$) implies $[(\gamma * \alpha * \gamma^-) *(\gamma_i * \tilde \alpha_i^- * \gamma_i^-)]_{\pi_1}\in \L(X, 3r, \pi_1)$;
 	\item the fact that group $\L(X, 3r, \pi_1)$ is a normal subgroups of $\L(X, fin, \pi_1)$ by Proposition \ref{PropIneq},
\end{itemize}
 and conclude $(\gamma * \alpha * \gamma^-) \in (\gamma_i * \alpha_i * \gamma_i^- )\L(X, 3r, \pi_1)$ and $(\gamma_i * \alpha_i * \gamma_i^-) \in (\gamma * \alpha * \gamma^- )\L(X, 3r, \pi_1)$. 

Note that the construction in the previous paragraph is symmetric with respect to $\alpha$ and $\alpha_i$, i.e., for each path $\gamma_i$ \new{of finite length} from $\bullet$ to $\alpha_i$ we can also construct an appropriate prolongation $\gamma$ \new{of finite length} from $\bullet$ to $\alpha$ so that the conclusions of the previous paragraph hold. We conclude that $G_{\alpha_i} \L(X, 3r, \pi_1) =G_{\alpha} \L(X, 3r, \pi_1)$.

The last two statements follow from Propositions \ref{PropPrePers}(2) and \ref{PropIneq}. 
\end{proof}

\section{The structure of $\pi_1$-persistence}
\label{SectStruc}
In this section we  show how critical values are generated by geodesic circles, i.e. Theorem \ref{ThmPerBasis}. Our approach essentially combines the correspondence results of Section \ref{SectFundGroup} and the limiting argument of Section \ref{SectLimit}. Along the way we get a number of results on the structure of persistence: local finiteness of critical values, direction of critical values, etc. We conclude with two subsections. The first one introduces a decomposition into lassos of Proposition \ref{PropLassoDecomp}, which encodes the persistence of an element. The second is Theorem \ref{ThmIntrLoc}, which is a localization result stating that a geodesic circle, corresponding to a critical value of persistence, is encoded in the critical triangle of persistence (see Definition \ref{DefCritTri} for the definition of critical triangles). Thus the local feature (geodesic circle) inducing a critical value is located by the critical triangle. 

We start with the main technical argument of the section.

\begin{Theorem} \label{PropFin}
 Suppose $X$ is compact and geodesic,  and let $q$ be a critical value for $\pi_1$ persistence of $X$ via open Rips filtration.
 \begin{description}
 \item [a] For each $r>0$ there exist only finitely many critical values for $\pi_1$ persistence via open Rips complexes, which are greater than $r$. 
\item [b] Let $q<p$ be a pair of consecutive critical values. Group $\L(X, 3p, \pi_1)$ is generated by  $\L(X, 3q, \pi_1)$ and a collection of geodesic $3q$-lassos.
 \item [c] There exists a geodesic circle of length $3q$. 
 \item [d] Let $q<p$ be a pair of consecutive critical values. Group $\S(X, 3p/2, \pi_1)$ is generated by  $\S(X, 3q/2, \pi_1)$ and the same collection of geodesic $3q$-lassos that was used in \textbf{b}.
\end{description}
 \end{Theorem}

\begin{Remark}
 It is easy to find an example demonstrating that a collection of geodesic circles in Theorem \ref{PropFin} needs not be finite. For example, consider $X$ to be the wedge of two circles of circumferences $1$ and $2$ with $\bullet$ being the common point. Let $\alpha$ denote the smaller circle. Then $\L(X, 3/2, \pi_1)$ is generated by a countable collection of $\alpha$-lassos and is not finitely generated.  Putting it differently, the normal subgroup of $\pi_1(X,\bullet)$ generated by  $\alpha$ is not finitely generated. In fact, one can observe that $\L(X, 3/2, \pi_1)$ is isomorphic to the kernel of the homomorphism defined on the free group on two elements $F_{a,b}\to \ZZ$ mapping $a\mapsto 0$ and $b \mapsto 1$.
 \end{Remark}

\begin{proof}
We first prove \textbf{a}. Suppose there are infinitely many critical values greater than $r$. $T=2 \diam(X)/3$ is its upper bound by Proposition \ref{PropPrePers}. We take a strictly monotonic subsequence  $t_1, t_2, \ldots\to t\geq r$ of such critical values.  By Proposition \ref{PropLambda} there exists a sequence of loops $(\alpha_i)_{i>1}$ with the following properties:
 \begin{enumerate}
 \item the lengths of circles form a strictly monotonic (increasing or decreasing, depending on the sequence of $t_i$'s) sequence. For example we may take $\length (\alpha_i)>3r$ to be strictly between $3(t_i+t_{i+1})/2$ and $3(t_i+t_{i-1})/2$;
 \item for each $i$ there exists a path $\gamma_i$ \new{of finite length} from $\bullet$ to some point on $\alpha_i$ so that $\gamma_i * \alpha_i * \gamma_i^-$ can't be expressed (up to homotopy) as a finite concatenation of geodesic lassos of length smaller than $\min \{3(t_i+t_{i+1})/2,3(t_i+t_{i-1})/2\}$.
\end{enumerate}
It follows from Lemma \ref{LemmaConvergence} that there exist $j,k\in\NN$ so that $t_j<t_k$ and $\gamma_k * \alpha_k * \gamma_k^-=\gamma_j * \alpha_j * \gamma_j^- \ \L(X, 3r, \pi_1)$, a contradiction to (2). This proves \textbf{a}.

We now turn our attention to \textbf{b}. Suppose $q$ is the only critical value on the interval $[q-\eps, q+\eps]$. Assume group $\L(X, 3p, \pi_1)$ is not generated by  $\L(X, 3q, \pi_1)$ and a collection of $3q$-lassos.
That means that for all $i\in \NN$ we can find a lasso $\gamma_i* \alpha_i * \gamma_i^-$  \new{of finite length} (based at $\bullet$) so that $3q<\length(\alpha_i)< 3q + \eps/i$ and $[\gamma_i* \alpha_i * \gamma_i^-]_{\pi_1}$ can't be obtained using 
$L(X, 3q, \pi_1)$ and $3q$-lassos. By Lemma \ref{LemmaConvergence} there exist a lasso $\gamma * \alpha * \gamma^-$ \new{of finite length} with $\length(\alpha)\leq 3q$ and $i$ so that $\gamma_i* \alpha_i * \gamma_i^- \simeq (\gamma * \alpha * \gamma^-)L(X, 3q, \pi_1)$, a contradiction. 
 Also, any $3q$-lasso which is not a geodesic lasso can be obviously decomposed, proving that its homotopy class is contained in $\L(X, 3q, \pi_1)$ hence \textbf{b} holds. Statement \textbf{c} follows from \textbf{b} and the fact that $q$ is a critical value.
 
As last we prove \textbf{d}. 
By  Theorems \ref{ThmPerSpa} and \ref{ThmPerCirc} we have
$$
\{\pi_1(X, \bullet)/\S(X, 3p/2, \pi_1)\}_{r>0}
\cong
\{\L(X,fin, \pi_1)/
\L(X, 3p, \pi_1)
\big\}_{r>0},
$$  
with the corresponding isomorphism $F$ being induced by samplings and fillings.   Isomorphism $F$ may be chosen to be identity on geodesic circles for appropriate parameters.
By  \textbf{b}, $\L(X, 3p, \pi_1)$ is generated by $\L(X, 3q, \pi_1)$ and $A$, the later being a collection of geodesic $3q$-lassos. Using the nature of $F$ mentioned above, we conclude that $\S(X, 3p/2, \pi_1)$ contains $N$, which is the subgroup of $\pi_1(X, \bullet)$ generated by $\S(X, 3q/2, \pi_1)$ and $A$. We claim that $\S(X, 3p/2, \pi_1)$ equals $N$. Choose $[\beta]\in \S(X, 3p/2, \pi_1)$, and let $\beta'$ be a filling of a $(3q/2)$-sample of $\beta$. Using the property of $F$, $[\beta']$ is contained in a subgroup generated by $\L(X, 3q, \pi_1)$ and $A$. Applying $F^{-1}$ (which again fixes geodesic circles), we see that $\beta'$ is contained in a subgroup generated by $\S(X, 3q/2, \pi_1)$ and $A$. By Proposition \ref{PropKer}, $\beta$ is also contained in a subgroup generated by $\S(X, 3q/2, \pi_1)$ and $A$. We conclude that $\S(X, r, \pi_1)$ equals $N$.
 \end{proof}
 
 \begin{Remark}\label{RemNecess}
Most of the statements of this section do not hold if $X$ is not compact. For example consider a geodesic space, obtained as a wedge of circles of circumference $1+1/n$ for all $n\in \NN$. Note that $1/3$ is a critical value of the $\pi_1$-persistence of $X$ via open Rips complexes, but there is no geodesic circle of length $1$. Furthermore, the collection of critical values greater than $1/3$ is not discrete. 
\end{Remark}

The following is a strengthening of Theorem \ref{PropFin} for SLSC spaces. 

\begin{Corollary}\label{CorFinCritValue}
  If $X$ is compact, geodesic, and SLSC then 
\begin{enumerate}
 \item  there exist only finitely many critical values for $\pi_1$ persistence of $X$ via open Rips filtration;
 \item   $\pi_1(X,\bullet)$ is generated by a finite collection of geodesic lassos; 
 \item if $q$ is a critical value for $\pi_1$ persistence of $X$ via open Rips filtration then there exist at most finitely many (and at least one) pairwise (freely) non-homotopic geodesic circles of length $3q$;
 \item there exist at most finitely many (freely) non-homotopic geodesic circles in $X$.
\end{enumerate}
\end{Corollary}

\begin{proof}
If $X$ is SLSC then there exists $M>0$ so that all critical values are greater than $M$.  (1) follows from Theorem \ref{PropFin}(a). (2) follows from Theorem \ref{PropFin}(b) and Corollary \ref{CorFG}.

In order to prove (3) suppose the contrary, that there exist infinitely many pairwise (freely) non-homotopic geodesic circles $\alpha_i$ of length $3q$. Apply Lemma \ref{LemmaConvergence} to the collection of loops $\{\alpha_i\}_i$ and use the fact that for a SLSC space there exists $r>0$ so that $\L(X, 3r, \pi_1)=0$, to conclude that  for infinitely many indices $j$ the collections $G_{\alpha_j}$ coincide, a contradiction.

The proof of (4) is similar to that of (3). The lengths of non-contractible geodesic circles are between $r$ (as defined in the proof of (3)) and $2 \diam(X)$, and so we can apply Lemma \ref{LemmaConvergence} again.
\end{proof}

\begin{Corollary}\label{CorDirCrit}
 Suppose $X$ is compact and geodesic. Then all critical values for $\pi_1$ persistence via open Rips filtration are right critical values. 
\end{Corollary}

\begin{proof}
Follows from Theorem \ref{PropFin}(\textbf{a}). Given a critical $c$ value choose an interval $[c-\eps, c+\eps]$ containing no other critical value and choose $0<r<c-\eps$. If $c$ was a left critical value then by Proposition \ref{PropPrePers}(1) there would exist an $r$-loop $L$, which would be $c$-null but not $(c-\eps/2)$-null. Since any $c$-nullhomotopy requires only a finite amount of triangles in $\Rips(X, c)$, $L$ would also be $c'$-null for some $c'\in (c-\eps/2,c)$. Thus there would be a critical value on $[c-\eps/2,c']$, a contradiction.
 \end{proof}

\begin{Remark}\label{RemNecess1}
Corollary \ref{CorDirCrit}  does not hold if $X$ is not compact. For example consider a geodesic space, obtained as a wedge of circles of circumference $2-1/n$ for all $n\in \NN$. Note that $2$ is a right critical value of the $\pi_1$-persistence of $X$ via open Rips complexes.
\end{Remark}

\begin{Proposition}\label{PropLassoGens}
Suppose $X$ is compact and geodesic. 
The fundamental group of each Rips complex is generated by a finite collection of geodesic lassos. Each kernel of a bonding map is a normal closure of a finite collection of geodesic lassos.  
\end{Proposition}

\begin{proof}
  Choose $r>0$. By Theorem \ref{PropFin} (a)
 there are only finitely many critical values $c_1 < c_2 < \ldots < c_k$ larger or equal to $r$.
 By Theorem \ref{PropFin} (b) $\L(X, c_k+1, \pi_1)$ is generated by 
 $\L(X, c_1, \pi_1)$ and a collection of geodesic lassos. 
By Proposition \ref{PropLambda} we conclude that $\pi_1(\Rips(X,r))\cong \L(X, c_k+1, \pi_1)/\L(X, c_1, \pi_1)$  is generated by geodesic lassos. By Corollary \ref{CorFG} this collection can be chosen to be finite (see the next paragraph for further technical details). 
  
 The second statement follows similarly from  Proposition \ref{PropLambda} and Theorem \ref{PropFin}(b). It suffices to provide a proof for a kernel of the bonding map going past one critical value only, say $c_1$. Consider for example the kernel of $\pi_1(\Rips(X, c_1/2), \bullet)\to \pi_1(\Rips(X, (c_1+c_2)/2), \bullet)$. It is generated by all $3c_1$-lassos. However, by Lemma \ref{LemmaConvergence} we only need a finite collection of geodesic circles $\alpha_1, \alpha_2, \ldots, \alpha_k$ of length $3 c_1$, so that all possible $\alpha_i$-lassos generate the kernel. Since for each single $i$, all  $\alpha_i$-lassos are conjugates to each other in the fundamental group, it follows that the normal closure of one $\alpha_i$-lasso  contains all other $\alpha_i$-lassos. Consequently, the kernel is the normal closure of a collection, which contains a single $\alpha_i$-lasso for each $i$.
\end{proof}

\begin{Theorem}\label{ThmPerBasis}
 [Persistence-basis correspondence] 
 Suppose $X$ is a compact geodesic space. Then there exist geodesic lassos $\{\beta_i\}_{i\in J}$, with $\beta_i$ being an $l_i$-lasso, so that the following isomorphism holds: 
 $$\big\{\pi_1(\Rips(X,r),\bullet)\big\}_{r>0}\cong\big\{\L(X,fin, \pi_1)/
\{\beta_i\big\}_{l_i<r}
\}_{r>0},$$
with bonding maps of the right-side persistence being the natural quotient maps. The set of critical points coincides with $\{l_i/3\}_{i\in J}$. For each critical value $c$ there exist only finitely many indices $j$ for which $l_j=3c$.

Furthermore, if $X$ is SLSC then $J$ can be chosen to be finite and
 $$\big\{\pi_1(\Rips(X,r),\bullet)\big\}_{r>0}\cong\big\{\pi_1(X, \bullet)/
\{\beta_i\big\}_{l_i<r}
\}_{r>0}.$$
\end{Theorem}

\begin{proof}
 Follows from Proposition \ref{PropLassoGens}.
\end{proof}

\begin{Remark}
 The question whether lassos $\alpha_i$ in Theorem \ref{ThmPerBasis} form a basis of $\pi_1(X, \bullet)$ and corresponding questions will be studied in a separate paper. It is easy to see though that the lassos may not be independent in the group and thus do not form the least numerous generating set.
\end{Remark}

\begin{Proposition}
 Suppose $X$ is a geodesic space of diameter $D$ and $G$ is an Abelian group. There are no geodesic loops of length more than $2D$ and for $r>2D/3$ groups $\pi_1(\Rips(X,r),\bullet)$ and $H_1(\Rips(X,r),G)$ are trivial. 
\end{Proposition}

\begin{proof}
 In a geodesic loop of length more than $2D$ there are points at distance more than $D$, a contradiction. The second statement follows by Theorem \ref{ThmPerBasis}. The last part follows by the Hurewicz Theorem and the Universal Coefficients Theorem.
\end{proof}

\subsection{Lasso decompositions}

\begin{Definition}
Suppose $X$ is compact and geodesic. Given an element  $[\alpha]_{\pi_1}\in\pi_1(X, \bullet)$ and $K>0$, a \textbf{lasso} $K$-\textbf{decomposition} is a representation of $[\alpha]_{\pi_1}$ in the form of a finite concatenation of $l_i$-lassos  with $l_i\leq K, \forall i$. The \textbf{critical value} $\crit[\alpha]_{\pi_1}$ is one third of the infimum of all values $K'$, for which there exists a lasso 
$K'$-decomposition. 
\end{Definition}

\begin{Proposition}
 \label{PropLassoDecomp}
 Suppose $X$ is compact, geodesic and SLSC. Choose $[\alpha]_{\pi_1}\in\pi_1(X,\bullet)$.  Then the critical value $\crit[\alpha]_{\pi_1}$ is positive and the following holds:
 $$
 \lambda_r^{\pi_1}([\alpha]_{\pi_1}) = 0 \quad \Leftrightarrow \quad r > \crit[\alpha]_{\pi_1}.
 $$
\end{Proposition}

\begin{proof}
Since $X$ is SLSC and compact,  map $\lambda_{r'}^{\pi_1}$ is an isomorphism for small $r'$, implying $\crit[\alpha]_{\pi_1}>0$, and trivial for large $r'$. By  Theorem \ref{ThmPerCirc} and Proposition \ref{PropLength}, the following holds:  $ \lambda_r^{\pi_1}([\alpha]_{\pi_1}) = 0$ if and only if there exists a lasso $r'$-decomposition of $[\alpha]_{\pi_1}$ for some $r'<r$. By Corollary \ref{CorFinCritValue} there are only finitely many critical values $c_1, c_2, \ldots, c_k$ of persistence and all of them are right critical values, implying the interval $\{r\mid  \lambda_r^{\pi_1}([\alpha]_{\pi_1}) = 0\}$ is an open interval. These are the only parameters at which $ \lambda_r^{\pi_1}([\alpha]_{\pi_1})$ may become trivial. Given a critical value $c_m$ the corresponding kernel is generated by a finite collection $A_m$ consisting of some (equivalently, all) geodesic $3c_1$-lassos, $3c_2$-lassos, .. and $3c_m$-lassos, by Proposition \ref{PropLassoGens}. Hence $ \lambda_r^{\pi_1}([\alpha]_{\pi_1})$ becomes trivial at $c_m$ if and only if it can be expressed by elements of $A_k$, meaning that $\crit[\alpha]_{\pi_1}=c_m$.
\end{proof}

\subsection{Localization theorem}

We conclude the section with an intrinsic version of a localization theorem. In TDA, a localization problem is a hard problem of finding geometric features (in a space) corresponding to critical values, for example loops $\beta_i$ of Theorem \ref{ThmPerBasis}. It turns out that in the intrinsic setting the optimal representatives  (in terms of a  $r$-lasso representative with the smallest $r$) can be obtained from the vertices of the critical triangles by fillings.

\begin{Definition}
 \label{DefCritTri}
 Suppose $X$ is a compact geodesic space and $r>0$. A triangle $\sigma\in \cRips(X, r)$ with vertices  $x_1, x_2, x_3$ is \textbf{critical} if the kernel of the inclusion induced map $\pi_1(\Rips(X, r), \bullet)\to \pi_1(\Rips(X, r)\cup \sigma, \bullet)$ is nontrivial, where adding $\sigma$ to the Rips complex implies we attach it along some $(r/2)$-sample of some filling $\alpha_\sigma$ of the $2r$-loop $x_1, x_2, x_3, x_1$. That means that we essentially  attach a hexagon, which is topologically equivalent to attaching a disc along the filling of $x_1, x_2, x_3, x_1$. These complications arise since the boundary edges of $\sigma$ are not in $\Rips(X, r)$. Note that in this case $\pi_1(\Rips(X, r), \bullet)$ and  $\pi_1(\Rips(X, r)\cup \sigma, \bullet)$ may still be isomorphic.
\end{Definition}

\begin{Theorem}\label{ThmIntrLoc}
 [Intrinsic Localization theorem]
 Suppose $X$ is a compact geodesic space and $r>0$. If $\sigma\in \cRips(X, r)$ with vertices  $x_1, x_2, x_3$ is a critical triangle then:
\begin{enumerate}
 \item points $x_1, x_2, x_3$ are at pairwise distance $r$;
 \item no $\alpha_\sigma$-lasso is contained in $\L(X, 3r, \pi_1)$, where $\alpha_\sigma$ is a filling of $x_1, x_2, x_3, x_1$ and has length $3r$.
\end{enumerate}
Furthermore, the kernel of $\pi_1(\Rips(X, 3r), \bullet) \to \pi_1(\Rips(X, 3r'), \bullet)$, where $r<r'$ are consecutive critical values, is generated by all $\alpha'$-lassos, where $\alpha'$ ranges through  fillings of vertices of each of the critical triangles in $\cRips(X, 3r)$. 

Conversely, if $\alpha$ is a loop in $X$ of length $3r$ so that some (equivalently, each) $\alpha$-lasso is not contained in $\L(X, 3r, \pi_1)$, then $\alpha$ is a geodesic circle, $r$ is a critical value and any three equidistant points on $\alpha$ form vertices of a critical triangle $\sigma\in \cRips(X, 3r)$.  
\end{Theorem}

\begin{proof}
Suppose $\sigma$ is a critical simplex. If points $x_1, x_2, x_3$ were not at the  pairwise distance $r$, then the $(r/2)$-loop along which we are attaching $\sigma$ to $\Rips(X, r)$ would be of length less than $3r$, hence would be contractible. Such an attaching could not change the fundamental group hence (1) holds.  A similar argument proves (2). The third claim follows from the representation of the kernels in question by geodesic lassos in  Theorem \ref{ThmPerBasis}.

The claims of the second paragraph hold by Theorem \ref{ThmPerCirc}.
\end{proof}

\section{The structure of $H_1$-persistence and minimal generators}
\label{SectHStruc}

In this section we provide homological versions of results of  Sections \ref{SectLimit} and \ref{SectStruc}. Proofs are omitted, as they are a simplification of the ones for the fundamental group and often follow from them directly using the Hurewicz Theorem and the Universal Coefficients Theorem. Instead we provide references to the corresponding statements on the fundamental groups. We also include Subsection \ref{SubsectMin} containing the Minimal Generators Theorem in the case of $G$ being a field. It states that the $1$-dimensional PD  encodes the lengths of the elements of the  shortest generating base.

\begin{Lemma}\label{LemmaHConvergence}
[A version of Lemma \ref{LemmaConvergence}]
 Suppose $X$ is a compact geodesic space, $r>0$ and $G$ is an Abelian group. For each $i\in \NN$ let $\alpha_i$ be a  loop in $X$ of length $l_i$ and let $l_i\to l>0$ as $i \to \infty$. Then there exists a loop $\alpha$ of length at most $l$ in $X$ and an infinite set $J\subset \NN$ so that 
$$
[\alpha_j]_G \ \L(X, 3r, G) =[\alpha]_G \ \L(X, 3r, G), \quad \forall j\in J.
$$
(If all $\alpha_i$ are geodesic circles then  the length of $\alpha$ equals $l$.)
\end{Lemma}

\begin{Theorem}
 \label{PropHFin}
 [A version of Theorem \ref{PropFin}, Corollary \ref{CorDirCrit} and Proposition \ref{PropLassoGens}]
 Suppose $X$ is compact and geodesic,  and let $q$ be a critical value for $H_1(\_, G)$ persistence of $X$ via open Rips filtration.
 \begin{description}
 \item [a] For each $r>0$ there exist only finitely many critical values, which are greater than $r$. 
\item [b] Let $q<p$ be a pair of consecutive critical values. Group $\L(X, 3p, G)$ is generated by  $\L(X, 3q, G)$ and a finite collection of geodesic circles of length $3q$.
 \item [c] There exists a geodesic circle of length $3q$. 
 \item [d] All critical values are right critical values.
 \item [e] First homology groups of each Rips complex is generated by a finite collection of geodesic circles. Each kernel of a bonding map is a generated by a finite collection of geodesic circles.  
 \item[f] If $G=\FF$ is a field then the corresponding persistence is q-tame in a sense of \cite{Cha1}.
\end{description}
\end{Theorem}

\begin{proof}
Only the finiteness in \textbf{b} requires a comment. It follows from Lemma \ref{LemmaHConvergence}  that at most finitely many different classes in $\L(X, 3p, G)$ have as a representative a geodesic circle of length $3q$.
\end{proof}

\begin{Corollary}\label{CorHFinCritValue}
[A version of Corollary \ref{CorFinCritValue}]
  Let $G$ be an Abelian group. If $X$ is compact, geodesic, and $G$-SLSC then 
\begin{enumerate}
 \item  there exist only finitely many critical values for $H_1(\_, G)$ persistence of $X$ via open Rips filtration;
 \item   $H_1(X,G)$ is generated by a finite collection of geodesic circles; 
 \item if $q$ is a critical value for $H_1(\_, G)$ persistence of $X$ via open Rips filtration then there exist at most finitely many (and at least one) pairwise  non-homologous (in $H_1(X, G)$) geodesic circles of length $3q$;
 \item there exist at most finitely many  pairwise  non-homologous in ($H_1(X, G)$) geodesic circles in $X$.
\end{enumerate}
\end{Corollary}

\begin{Theorem}\label{ThmHPerBasis}
[A version of Theorem \ref{ThmPerBasis}]
 [Persistence-basis correspondence] 
 Suppose $X$ is a compact geodesic space and $G$ is an Abelian group. Then there exist geodesic circles $\{\beta_i\}_{i\in J}$ of length $l_i=\length(\beta_i)$, generating $H_1(X, G)$, so that the following isomorphism holds: 
 $$\big\{H_1(\Rips(X,r),G)\big\}_{r>0}\cong\big\{\L(X,fin, G)/
\{\beta_i\big\}_{l_i<r}
\}_{r>0},$$
with bonding maps of the right-side persistence being the natural quotient maps. The set of critical points coincides with $\{l_i/3\}_{i\in J}$. For each critical value $c$ there exist only finitely many indices $j$ for which $l_j=3c$.

Furthermore, if $X$ is $G$-SLSC then $J$ can be chosen to be finite and
 $$\big\{H_1(\Rips(X,r),G)\big\}_{r>0}\cong\big\{H_1(X, G)/
\{\beta_i\big\}_{l_i<r}
\}_{r>0}.$$
\end{Theorem}

\subsection{Minimal loop decompositions}

\begin{Definition}
Suppose $G$ is an Abelian group and $X$ is compact and geodesic. Given an element  $[\alpha]_{G}\in H_1(X, G)$ and $K>0$, a \textbf{loop} $K$-\textbf{decomposition} is a representative of $[\alpha]_{G}$ in the form of a finite sum of loops of length $l_i$  with $l_i\leq K, \forall i$. The \textbf{critical value} $\crit[\alpha]_{G}$ is one third of the infimum of all values $K'$, for which there exists a loop 
$K'$-decomposition. 
\end{Definition}

\begin{Proposition}
 \label{PropHLassoDecomp}
 Suppose $G$ is an Abelian group and $X$ is compact, geodesic and $G$-SLSC. Choose $[\alpha]_{G}\in H_1(X,G)$.  Then the critical value $\crit[\alpha]_{G}$ is positive and the following holds:
 $$
 \lambda_r^{G}([\alpha]_{G}) = 0 \quad \Leftrightarrow \quad r > \crit[\alpha]_{G}.
 $$
\end{Proposition}

\begin{proof}
The proof is essentially the same as that of Proposition \ref{PropLassoDecomp}.
\end{proof}

\subsection{Minimal generating set}
\label{SubsectMin}

Assume now (for this subsection) that $\FF$ is a field and $X$ is compact, geodesic and $\FF$-SLSC. We know that 
\begin{enumerate}
\item $H_1(X, \FF)=\FF^n$;
\item there are critical values (with possible repetitions) $C=\{c_1 \leq c_2 \leq \ldots \leq c_n\}$;
\item defining $k_r = \max_{c_i<r}i$ the persistence is of form $H_1(\Rips(X, r),\FF)= \oplus_{i=1}^{k_r} \FF\  \oplus_{j=k_r+1}^n \{0\}$ with the bonding maps being surjective on the corresponding summand, either identities or trivial;
\item $H_1(\Rips(X, r),\FF)= \oplus_{i=1}^{n} \II_{(0, c_i]}$.
\end{enumerate}

\begin{Definition}\label{DefMinBase}
A collection of loops $A=\{ a_1, a_2, \ldots a_n\}$ in $X$ with lengths $l_i=\length(a_i)$ satisfying $ l_i\leq l_j, \forall i<j$ is a base for $H_1(X, \FF)$ if elements of $A$ generate $H_1(X, \FF)=\FF^n$. Such base is a
\textbf{lexicographically minimal base} for $H_1(X, \FF)$ if the length vector  $\widetilde L=\{l_1, l_2, \ldots, l_n\}$ forms a lexicographically minimal vector amongst length vectors of all bases. 
\end{Definition}

\begin{Remark}
It follows from Propositions \ref{PropHLength}(2) and \ref{PropDiam}  that the concept of a lexicographically minimal base does not depend on the measure of the size of loops that we use: as in Section \ref{SectFundGroup} we can use either the radius of the smallest enclosing ball, length (as in Definition \ref{DefMinBase}) or the diameter. Consequently we could rephrase Theorem \ref{ThmLexicograph} with these alternative measures as well.
\end{Remark}

\begin{Proposition}
 Suppose $\FF$ is a field and $X$ is compact, geodesic and $\FF$-SLSC. Then there exists a lexicographically minimal base $A=\{ a_1, a_2, \ldots a_n\}$ in $X$ with lengths $l_i=\length(a_i)$ satisfying $ l_i\leq l_j, \forall i<j$. The base is not unique. However, all loops $a_i$ are geodesic circles.
\end{Proposition}

\begin{proof}
We can choose the base essentially in the same way as (linearly $\FF$-independent) loops $\beta_i$ in Theorem \ref{ThmHPerBasis}. If this base was not minimal then it would contradict the connection to persistence stated in Theorem \ref{ThmHPerBasis}. The base is not unique as the usual cylinder demonstrates. If any of the loops $a_i$ was not a geodesic circle we could decompose it by connecting two diametrically opposed points (with respect to circumference of $a_i$) at distance less than $l_i/2$ and thus obtain a lexicographically smaller base.
\end{proof}

\begin{Theorem}
 \label{ThmLexicograph}
 [Persistence-Minimal Base correspondence]
  Suppose $\FF$ is a field and $X$ is compact, geodesic and $\FF$-SLSC. Let $A=\{ a_1, a_2, \ldots a_n\}$ be a collection of loops in $X$  with lengths $l_i=\length(a_i)$ satisfying $ l_i\leq l_j, \forall i<j$  and $H_1(X, \FF)=\FF^n$. 
  
  Then $A$ is a lexicographically minimal base if and only if
    $$
  \big\{H_1(\Rips(X, r),\FF)\big\} _{r>0} = \big\{ H_1(X, \FF)/\{[a_1]_\FF, [a_2]_\FF\ldots, [a_{k_r}]_\FF \} \big\} _{r>0}\cong \bigoplus_{i=1}^n (\II_{(0, l_i/3]})_r,  
  $$
  where $k_r = \max_{(l_i/3)<r}i$. In such case  the corresponding critical values of persistence are $\{l_i/3\}_{i\in \{1, 2, \ldots, n\}}$. 
\end{Theorem}

Thus the lexicographically minimal base encodes the persistence diagram of $X$ along with the corresponding critical triangles: the lengths of generators encode death times and generators  correspond to critical triangles (see Figure \ref{FigExample}).

\subsection{Localization theorem}

\begin{Definition}
 \label{DefHCritTri}
 Suppose $X$ is a compact geodesic space, $G$ is an Abelian group and $r>0$. A triangle $\sigma\in \cRips(X, r)$ with vertices  $x_1, x_2, x_3$ is $G$-\textbf{critical} if the kernel of the inclusion induced map $H_1(\Rips(X, r), G)\to H_1(\Rips(X, r)\cup \sigma, G)$ is nontrivial, where adding $\sigma$ to the Rips complex implies we attach it along some $(r/2)$-sample of some filling $\alpha_\sigma$ of the $2r$-loop $x_1, x_2, x_3, x_1$.
 \end{Definition}

\begin{Theorem}\label{ThmHIntrLoc}
 [Intrinsic Localization theorem for homology]
 Suppose $X$ is a compact geodesic space and $r>0$. If $\sigma\in \cRips(X, r)$ with vertices  $x_1, x_2, x_3$ is a $G$-critical triangle then:
\begin{enumerate}
 \item points $x_1, x_2, x_3$ are at pairwise distance $r$;
 \item no $\alpha_\sigma$ is contained in $\L(X, 3r, G)$, where $\alpha_\sigma$ is a filling of $x_1, x_2, x_3, x_1$ and has length $3r$.
\end{enumerate}
Furthermore, the kernel of $H_1(\Rips(X, 3r), G) \to H_1(\Rips(X, 3r'), G)$, where $r<r'$ are consecutive critical values, is generated by all $\alpha'$, where $\alpha'$ ranges through  fillings of vertices of each of the critical triangles in $\cRips(X, 3r)$. 

Conversely, if $\alpha$ is a loop in $X$ of length $3r$ not contained in $\L(X, 3r, G)$, then $\alpha$ is a geodesic circle, $r$ is a $G$-critical value and any three equidistant points on $\alpha$ form vertices of a $G$-critical triangle $\sigma\in \cRips(X, 3r)$.  
\end{Theorem}

\begin{proof}
The proof is essentially the same as that of Theorem \ref{ThmIntrLoc}.
\end{proof}

\section{An example}
\label{SectExample}
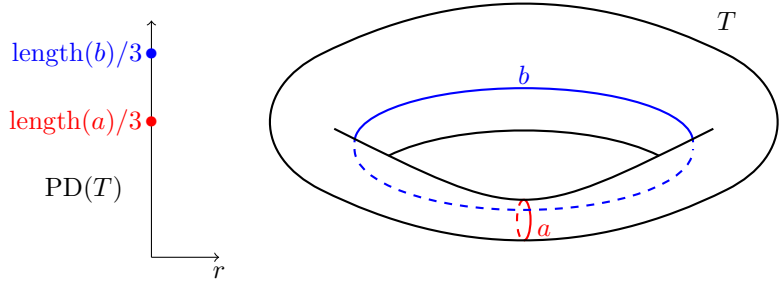
\begin{figure}
\begin{tikzpicture}[scale=.9]
\draw [red, thick](0,-.75) arc (-90:90:.1 and .3);
\draw [red, thick, dashed](0,-.15) arc (90:270:.1 and .3);
\draw [red](0.3, -.6) node {$a$};
\draw [blue, thick](2.5,.7) arc (0:180:2.5 and .8);
\draw [blue, thick, dashed](-2.5,.6) arc (180:360:2.5 and .9);
\draw [blue](0, 1.7) node {$b$};
\draw [thick] (-3,0) ..controls (-1,-1)and(1,-1)..
(3,0)..controls(4,.5)and(4, 1.5)..
(3,2)..controls(1,3)and(-1, 3)..
(-3,2)..controls(-4,1.5)and(-4, .5)..
cycle;
\draw [thick](-2.8,.9) ..controls (0,-.5)..(2.8,.9);
\draw  [thick](-2,.5) ..controls (-1,1)and(1,1)..(2,.5);
\draw[->] (-5.5, -1) to (-4.5, -1) node[below]{$r$};
\draw[->] (-5.5, -1) to (-5.5, 2.5) node[above]{};
\draw[red] (-5.5, 1) node {$\bullet$} node[left]{$\length(a)/3$};
\draw[blue] (-5.5, 2) node {$\bullet$}node[left]{$\length(b)/3$};
\draw (-6.5, 0) node{PD($T$)};
\draw (3, 2.5) node{$T$};

\end{tikzpicture}
\caption{A slightly deformed standard torus $T$ with shortest homology generators $a$ and $b$ on the right and its persistence diagram via open Rips complexes for any field on the left.}
\label{FigExample}
\end{figure}

Consider a two-dimensional torus which is slightly thinned at some region so that we have unique shortest geodesic circles $a$ and $b$ generating its homology as depicted on Figure \ref{FigExample}. Its persistence diagram is provided by the same figure. Now assume we remove an open disc $D$ far from $a$ and $b$, whose boundary  $c$ is of length much smaller than $a$. Denote the obtained geodesic space by $T'$. Note that the removal of  $D$ changes metric at some region a little bit, and makes $c$ a geodesic circle. However $a$ and $b$ remain unique shortest geodesic circles  generating the homology.

Then 
$$
\pi_1(\Rips(T',r),\bullet) = 
\begin{cases}
\ZZ * \ZZ, & r\leq  \length(c)/3\\
\ZZ \oplus \ZZ, & \length(c)/3 < r\leq  \length(a)/3\\
\ZZ, & \length(a)/3 <  r\leq  \length(b)/3\\
0, & \length(b)/3 <  r
\end{cases}
$$
 and 
$$
H_1(\Rips(T',r),\ZZ) = 
\begin{cases}
\ZZ \oplus \ZZ, & r\leq  \length(a)/3\\
\ZZ, & \length(a)/3 <  r\leq  \length(b)/3\\
0, & \length(b)/3 <  r
\end{cases}
$$
Note that the critical values, which correspond to geodesic circles by the results of this paper, do not coincide in the provided cases. However, geodesic circle $c$ is still detectable by higher-dimensional homology (see the last section of Future Work) even though it is nullhomologous.

\section{\v Cech complexes}
\label{SectCech}

In this section we prove that  Rips and \v Cech complexes of a geodesic space essentially (that is up to a multiplication of parameter $r$ by $3/4$) induce the same persistences. A connection between \v Cech complexes and the corresponding Spanier groups $\S$ was first proved in \cite{BF} for paracompact spaces using the partitions of unity as the connecting map. For our purposes however we need a connecting map in terms of samples/fillings (in the spirit of discretization) and we need the commutativity with the inclusion induced maps on the algebraic level. For this purpose we prove the required connection using (in our case more convenient and geometric) $\L$ groups.

We will make use of the notation introduced in Definition \ref{DefRLoop} (1)-(6) in the \v Cech setting as well, as $1$-skeleta of $\Rips(X,r)$ and $\Cech(X,r/2)$ coincide.

\begin{Def}\label{CechLoops}
 An $r$-loop is \v Cech $r$-\textbf{null} if it is contractible in $\Cech(X,r/2)$. Two $r$-loops are \v Cech $r$-\textbf{homotopic}, if they are homotopic in $\Cech(X,r/2)$. The corresponding simplicial homotopy in $\Cech(X,r/2)$ is referred to as \v Cech $r$-\textbf{homotopy}. If the second $r$-loop is constant we also call it \v Cech $r$-\textbf{nullhomotopy}. 
\end{Def}

The following proposition provides the required properties from Proposition \ref{PropRips} in the \v Cech setting.

\begin{Prop}
\label{PropCech}
 Suppose $X$ is geodesic, $\alpha \colon [0,a]\to X$ is a loop in $X$ and fix $0<r$. Then the following hold:
 \begin{enumerate}
	\item  any two $r$-samples of $\alpha$ are \v Cech $r$-homotopic;
	\item any $r$-sample of a loop $\alpha$ of length less than $2r$ is \v Cech $r$-null;
	\item if loops $\alpha$ and $\alpha' \colon [0,a']\to X$ are homotopic, then any two of their $r$-samples $L$ and $L'$ are \v Cech $r$-homotopic (the statement holds for both based and unbased versions);
\end{enumerate}
\end{Prop}

\begin{proof}
 The proof follows the proof of Proposition \ref{PropRips} (4)-(6) closely. 

 (1) As in the proof of Proposition \ref{PropRips} (4) it suffices to show that adding a point $p$ between consecutive points $x_i$ and $x_j$ of an $r$-sample $L$ results in an $r$-loop, which is \v Cech $r$-homotopic to $L$.  The \v Cech $r$-homotopy is induced by a $2$-simplex $[x_i, x_j, p] \in \Cech(X,r/2)$, whose existence is witnessed by the midpoint between $x_i$ and $x_j$.
 
 (2) By (1) it suffices to prove the claim for any $r$-sample. Assume $\alpha \colon [0,a]\to X$ is a parametrization with the natural parameter. Choose an $r$-sample of the form $x_0, x_1, x_2, x_3, x_4=x_0$  with $x_i=\alpha(i \cdot a/4)$.
Then $L$ is \v Cech $r$-nullhomotopic via $2$-simplices $[x_0,x_1,x_2]$ and $[x_0,x_3,x_2]$ witnessed by points $x_1$ and $x_3$ in $\Cech(X,r/2)$ respectively. See the right side of Figure \ref{FigAA} for the corresponding configuration.
 
 (3) The proof of Proposition \ref{PropRips} (6) suffices as each  $2$-simplex $[a,b,c]$ used in that $r$-nullhomotopy is contained in some open ball of radius $r/2$, hence the center of that ball witnesses the existence of $[a,b,c]$ in $\Cech(X,r/2)$. 
\end{proof}

The following is a variant of Proposition \ref{PropNull} for the \v Cech complexes.

\begin{Proposition}\label{PropNullCech}
 Suppose $X$ is a geodesic space, $r>0$, $L$ is an $r$-loop and $\alpha$ is a filling of $L$. 
\begin{enumerate}
 \item If $L$ is based then the following holds:  $L$ is \v Cech $r$-null if and only if $[\alpha]_{\pi_1} \in \L(X, 2r, \pi_1)$.
 \item For any Abelian group $G$ the following holds:  $0=[L]_G\in \Cech(X,r/2)$ if and only if  $[\alpha]_G\in \L(X,2r,G)$.
\end{enumerate}
  \end{Proposition}
  
\begin{proof}
Again, we will prove (1) only. 

 \begin{figure}
\begin{tikzpicture}
\coordinate (A) at (-1,0);
\coordinate (B) at (-5, 1);
\coordinate (C) at (-5, 3);
\coordinate (D) at (-1, 3);
\coordinate (A1) at (4,0);
\coordinate (B1) at (0, 1);
\coordinate (C1) at (0, 3);
\coordinate (D1) at (4, 3);
\coordinate (E) at (2, 2);
\coordinate (F) at (1, 1);
%
\draw [gray, top color=blue,bottom color=white, fill opacity =.3] (A) -- (B) -- (C) -- cycle;
\draw [gray, bottom color=blue,top color=white, fill opacity =.3] (A) -- (D) -- (C) -- cycle;
\draw [gray, top color=blue,bottom color=white,  opacity =.1, fill opacity =.1] (A1) -- (B1) -- (C1) -- cycle;
\draw [gray, bottom color=blue,top color=white,  opacity =.1, fill opacity =.1] (A1) -- (D1) -- (C1) -- cycle;
%
\draw[red, thick](A1) to  [bend left=25] node[below]{$r$}(B1);
\draw[red, thick](A1)to [bend left=15]node[left]{$r$}(C1);
\draw[red, thick](C1) to [bend left=15]node[left]{$r$} (B1);
\draw[red, thick](A1) to [bend left=25] node[right]{$r$}(D1);
\draw[red, thick](C1) to [bend left=25] node[above]{$r$}(D1);
\draw[red, thick, dashed](C1) to [bend left=25]node[below]{$r/2$}(E);
\draw[red, thick, dashed](A1)to [bend right=15]node[left]{$r/2$}(E);
\draw[red, thick, dashed](D1) to [bend left=15]node[below]{$r/2$}(E);
\draw[red, thick, dashed](A1) to [bend left=25](F);
\draw[red, thick, dashed](B1) to [bend left=25](F);
\draw[red, thick, dashed](C1) to [bend right=5](F);
\draw [->](2.8,2.5) arc (-30:240:.5);
\draw (2.3, 2.7) node {$<2r$};
\draw [fill] (A) circle [radius=0.1];
\draw [fill] (B) circle [radius=0.1];
\draw [fill] (C) circle [radius=0.1];
\draw [fill] (D) circle [radius=0.1];
\draw [fill] (A1) circle [radius=0.1];
\draw [fill] (B1) circle [radius=0.1];
\draw [fill] (C1) circle [radius=0.1];
\draw [fill] (D1) circle [radius=0.1];
\node [rectangle,  scale=.5, fill=black,draw]  at (E) {};
\node [rectangle,  scale=.5, fill=black,draw]  at (F) {};
\end{tikzpicture}
\caption{
A sketch of the situation of Proposition \ref{PropNullCech}. Given a configuration of abstract triangles on the left, construct an appropriate system of geodesics (red lines) of corresponding bounds on their lengths (as suggested by the red label) on the right. The dashed geodesics connect vertices to the triple intersection of the corresponding balls, arising from the existence of the triangles. Note that the decomposition into triangles on the left corresponds to the decomposition into loops of length less than $2r$ on the right.
}
\label{FigCech}
\end{figure}
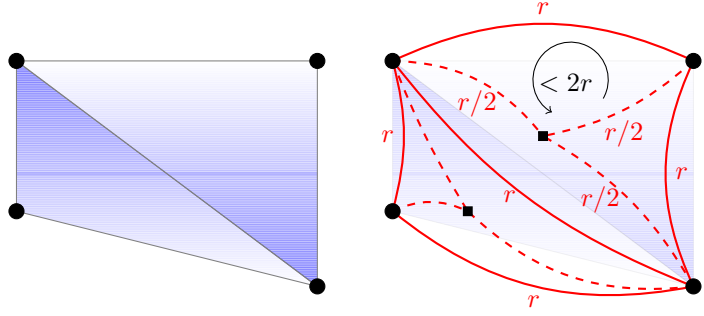
Suppose $L$, given by $\bullet=x_0, x_1, \ldots, x_k$, is \v Cech $r$-null. As in the proof of Proposition \ref{PropNull} that means that there is a \v Cech $r$-nullhomotopy given by a  triangulation $\Delta$ of a closed disc $D$. This means that if $[x,y]$ is an edge of $\Delta$ then $d(x,y)<r$. However, in this case we have an additional feature (see Figure \ref{FigCech}): for each triangle $T=[x,y,z]$ in $\Delta$ there is a point $w_T$ witnessing the existence of $T$  in $\C(X,r/2)$, i.e., $d(x,w)<r/2, d(y,w)<r/2$ and $d(z,w)<r/2.$ Hence we may take a subdivision $\Delta'$ of $\Delta$, obtained by adding all points $w_T$ and replacing each $T$ by the three triangles $[x,y,w], [x,z,w], [z,x,w]$. Note that the length of each of these three triangles as $r$-loops is less than $2r$. As in the proof of Proposition \ref{PropNull} this implies $[\alpha]_{\pi_1} \in \L(X, 2r, \pi_1)$.

Now suppose $\alpha \in \L(X, 2r, \pi_1)$. By Proposition \ref{PropCech} (3) it suffices to prove that any lasso of length less than $2r$ is \v Cech $r$-null. Assume therefore $\alpha = \gamma * \beta *\gamma^-$ where $\gamma$ is a based path and $\beta$ is a loop of length less than $2r$. It suffices to prove that any $r$-sample of $\beta$ is \v Cech $r$-null (in the unbased sense), which is true by Proposition \ref{PropCech}(2).
\end{proof}

\begin{Def}[A \v Cech version of Definition \ref{DefLambda}]
\label{DefLambdaC}
Suppose $r>0$, $X$ is a geodesic space and $G$ is an Abelian group. Maps
 $$
 \check\lambda_r^{\pi_1}\colon \pi_1(\Cech(X,r),\bullet)\to \L(X,fin, \pi_1)/
\L(X, 4r, \pi_1)
 $$
 and 
  $$
 \check\lambda_r^{G}\colon H(\Cech(X,r),G)\to \L(X,fin, G)/
\L(X, 4r, G)
 $$
 are defined by mapping an $2r$-loop to its filling.
\end{Def}

\begin{Proposition}[A \v Cech version of Proposition \ref{PropLambda}]
\label{PropLambdaC}
 For each $r>0$, each Abelian group $G$ and each geodesic space $X$, maps $ \check\lambda_r^{\pi_1}$ and $\check \lambda_r^{G}$are well defined isomorphisms. Furthermore, maps $\check\lambda^*_r$ commute with the inclusions $i_{p,q}\colon \Cech(X,p)\to \Cech(X, q)$ induced maps and the obvious quotient maps $\L(X,fin, *)/\L(X, 4p, *)\to \L(X,fin, *)/\L(X, 4q, *)$ for all $p<q$.
\end{Proposition}

\begin{proof}
The proof is essentially the same as that of Proposition \ref{PropLambda}.
Maps are well defined by Propositions \ref {PropDif} and  \ref{PropNullCech} and injective by Proposition \ref{PropNullCech}. It remains to prove they are also surjective. We provide the proof for $ \lambda_r^{\pi_1}$, the homological case can be proved in an identical manner. 
 
 Take any (based) loop $\alpha$ in $X$ of finite length. Choose an $r$-sample $L$ of $\alpha$ so that the length of $\alpha$ between each pair of consecutive points is less than $r$. It is easy to see (check for example a similar argument in Proposition \ref{PropDif}) that for each filling $\beta$ of $L$ we have $[\alpha * \beta^-]_{\pi_1}\in \L(X,2r,\pi_1)$, hence $L$ is mapped to the equivalence class of $\alpha$ in $ \L(X,fin, \pi_1)/\L(X, 4r, \pi_1)$ by $\lambda_r^{\pi_1}$.
 
The final claim about commutativity follows easily from the definitions, as in both direction we map a class represented by an $r$-loop $L$ to a class represented by a filling of $L$.
\end{proof}

\begin{Theorem} 
\label{ThmRipsCech}
[Rips-\v Cech correspondence Theorem]
For every geodesic space $X$ and for every Abelian group $G$ there are isomorphisms  of persistences
$$
\big\{\pi_1(\Cech(X,3r),\bullet)\big\}_{r>0}\cong \big\{\pi_1(\Rips(X,4r), \bullet) \}_{r>0}
$$
 and 
$$
\big\{H_1(\Cech(X,3r),G)\}_{r>0}\cong \big\{H_1(\Rips(X,4r), G)\}_{r>0}.
$$
\end{Theorem}

\begin{proof}
 By Propositions \ref{PropLambda} and \ref{PropLambdaC} both persistences are isomorphic to 
 $$\L(X,fin, *)/
\L(X, 12r, *).$$
\end{proof}

\section{Persistence for closed complexes, length spaces and applications}
\label{SectClosed}

\subsection{Persistence of compact geodesic spaces for filtrations by closed complexes}

A theory parallel to the one presented in this paper could also be developed using closed filtrations, i.e., filtrations by closed Rips and \v Cech complexes of compact geodesic spaces.  At some places though (for example in Proposition \ref{PropRips}(8))  additional technical work would be required to obtain equivalent results. Rather than writing down a parallel theory we deduce the results for closed complexes directly from the obtained results using the following arguments:
\begin{enumerate}
 \item \textbf{All bounding maps $\bar i_{p,q}\colon \pi_1(\cRips(X,p),\bullet)\to \pi_1(\cRips(X,q), \bullet)$ and their homological counterpart are surjective for all $p<q$.} The same holds for closed \v Cech filtrations. The proof is identical to the one of Proposition \ref{PropRips}(9).
 \item \textbf{For each $\eps>0$ there is an obvious $\eps$-interleaving between persistences $\{ \pi_1(\cRips(X,r),\bullet)\}$ and $\{ \pi_1(\Rips(X,r),\bullet)\}$ induced by inclusion maps.} The same holds for the homological counterparts and for closed \v Cech filtrations.  
\end{enumerate}

\begin{Remark}
Case (2) above  essentially implies that the observable structures, in the sense of \cite{Cha0}, of  open and closed filtrations are the same. Note that the observable structure in \cite{Cha0} was only defined for what is usually called a persistence module, that is a persistence where the underlying algebraic objects are vector spaces. That definition can be  generalized in an obvious way to persistences in a more general setting as considered in this paper.
\end{Remark}

From these arguments we may deduce the following conclusions for the Rips induced $\pi_1$- and $H_1$-induced persistences:
\begin{description}
\item [a] For each $r>0$, $\eps>0$ and Abelian group $G$ the maps induced on $\pi_1$ and $H_1(\_, G)$ by inclusions $\Rips(X,r)\hookrightarrow \cRips(X, r + \eps)$ and $\cRips(X,r)\hookrightarrow \Rips(X, r + \eps)$ are surjective. This follows easily from the fact that these induced maps commute with the corresponding surjective maps induced by inclusions $i_{p,q}$ and $\bar i_{p,q}$, see (1) above and Proposition \ref{PropRips}(9).
 \item[b] Critical values (and the corresponding kernels) of open and closed filtrations coincide, which follows  from (1),(2), \textbf{a}, Theorem \ref{PropFin}(1) and Theorem \ref{PropHFin}(1).
 \item [c] The direction of critical values is different: open filtrations induce right critical values while closed filtrations induce left critical values. To see this recall that each critical value $c$ for an open filtration is generated by geodesic $3c$-lassos (Theorem \ref{PropFin}) or geodesic circles of length $3c$ (Theorem \ref{PropHFin}). As in Proposition \ref{PropRips}(5) it is easy to see that any such geodesic lasso or geodesic circle is contractible when represented in $\cRips(X, 3c)$. For such representation we have to employ closed $3c$-samples, which are the same as $3c$-samples of Definition \ref{DefRLoop} with the only difference that strict inequalities $<$ are replaced by $\leq$.
\end{description}

Analogous conclusions also hold for \v Cech filtrations.
Note that argument (1) is required for conclusion \textbf{b}, which in general does not hold in higher dimensions or other settings. For example, the two-dimensional persistence of $S^1$ via open complexes is trivial while the one via closed is not by \cite{AA}. In this case open and closed Rips persistences do not determine each other.

We sum up the conclusions and their implications via Theorems \ref{ThmPerSpa}, \ref{ThmPerCirc}, \ref{ThmPerDiam}, \ref{ThmPerBasis}, \ref{ThmLexicograph}, and \ref{ThmRipsCech}  in the following theorem. 

\begin{Theorem}
 \label{ThmClosedPers}
 For every compact geodesic space $X$  there are isomorphisms  of persistences:
\begin{enumerate}
\item $\big\{\pi_1(\cRips(X,4r), \bullet) \}_{r>0}$
 \item $\big\{\pi_1(\cCech(X,3r),\bullet)\big\}_{r>0}$
 \item $\{\pi_1(X, \bullet)/\overline\S(X, 6r, \pi_1)\}_{r>0}$
  \item $\big\{\L(X,fin, \pi_1)/\overline\L(X, 12r, \pi_1)\big\}_{r>0}$
   \item $\{\pi_1(X, \bullet)/\overline\D(X, 6r, \pi_1)\}_{r>0}$
   \item $\big\{\pi_1(X, \bullet)/\{\beta_i\big\}_{l_i\leq r}\}_{r>0}$ for some collection of geodesic lassos $\{\beta_i\}_{i\in J}$, with each $\beta_i$ being an $l_i$-lasso, whose normal closure is $\pi_1(X, \bullet)$,
\end{enumerate}
 where $\overline \S, \overline L$ and $\overline D$ are closed versions of $\S, \L$ and $\D$, i.e., 
\begin{itemize}
 \item $\overline \S(X, r, *)$ is generated by loops of size at most $r$ or by $U_d$-lassos with $d\leq r$ (see Definition \ref{DefUr});
 \item $\overline \L(X, r, *)$ is generated by loops of length at most $r$ or by $l$-lassos with $l \leq r$ (see Definition \ref{DefL});
\item $\overline \D(X, r, *)$ is defined analogously using Definition \ref{DefDr}.
\end{itemize}
 
Analogous isomorphisms hold for $H_1$-persistence for each  Abelian group $G$. Furthermore, if $G=\FF$ is a field and $X$ is $\FF$-SLSC then the corresponding $H_1$-persistence is also isomorphic to,
$$
\bigoplus_{i=1}^n (\II_{(0, 12 l_i)})_r,
$$
where $A=\{ a_1, a_2, \ldots a_n\}$ is a lexicographically minimal base for $H_1(X, \FF)$  with lengths $l_i=\length(a_i)$.
\end{Theorem}

\subsection{Application: metric graphs}
A metric graph is a geodesic space obtained by assigning to each edge of a finite graph a positive length. For details see \cite{7A}, where the authors compute their $1$-dimensional intrinsic persistence diagram via open \v Cech complexes using coefficients in a field. Using the results of this paper we may extend their results to all fields (and groups), all complexes and provide the direction of the corresponding critical values, i.e., we provide all possible corresponding decorated persistence diagrams (see \cite{Cha0} for an introduction to decorated persistence diagrams). In fact we could even express the $\pi_1$-persistence.

\begin{Theorem}
[A generalization Theorem 1.1 of \cite{7A}]
 Suppose $X$ is a metric graph, $\FF$ is a field and let $A=\{ a_1, a_2, \ldots a_n\}$ be a lexicographically minimal base for $H_1(X, \FF)$  with lengths $l_i=\length(a_i)$. Then the $H_1(\_, \FF)$-persistence of $X$ is:
\begin{itemize}
 \item $\bigoplus_{i=1}^n (\II_{(0, l_i/3]})_r$ for persistence via open Rips filtration,
  \item $\bigoplus_{i=1}^n (\II_{(0, l_i/3)})_r$ for persistence via closed Rips filtration,
   \item $\bigoplus_{i=1}^n (\II_{(0, l_i/4]})_r$ for persistence via open \v Cech  filtration,
    \item $\bigoplus_{i=1}^n (\II_{(0, l_i/4)})_r$ for persistence via closed \v Cech  filtration.
\end{itemize}
\end{Theorem}

\subsection{Generalizations to length spaces and other relaxations}

The compactness of $X$ was only needed for results from Section \ref{SectStruc} on, excluding Section \ref{SectCech}. In a non-compact case the minimal generators of Theorems \ref{ThmPerBasis} and \ref{ThmLexicograph} may not exist (as in $\RR^2 \setminus \overline B((0,0),1) $) and the set of critical values may be any subset of positive real numbers, as demonstrated by the wedge of circles of all positive circumferences. Also, there might be no difference between the type of  individual critical values when using open or closed filtrations.

The results up to  Section \ref{SectLimit} and of Section \ref{SectCech} also hold for length spaces, i.e., spaces, where the distance between each pair of points is the infimum of the lengths of paths between them. One could use the same arguments as presented in this paper to develop a theory for them. This would only be interesting for  non-compact spaces as compact length spaces are geodesic by the Hopf-Rinow Theorem.

\section{Discussion and future work}
\label{SectFW}
The theory presented in this paper serves as a basis for a substantial future work:
\begin{itemize}
 \item \textbf{Stability and approximation}: In the subsequent paper \cite{ZV1} we show that the $1$-dimensional intrinsic persistence may be well approximated by finite samples and is much more stable than the standard Stability Theorem implies. In fact, it turns out that for a locally contractible compact geodesic space there exists a finite subset whose persistence determines the persistence of the underlying space thus eliminating the limiting argument.
  \item \textbf{Extracting geometric features from higher dimensional persistence}: By the theory of this paper each critical value corresponds to a geodesic circle. Using results from \cite{AA} we may prove that at least in a nice setting any geodesic circle will generate essential $S^3$ in the corresponding complex, detectable by any persistent homology. In particular, higher dimensional homology with coefficients in any field may provide us with all critical points for $\pi_1$-persistence, thus making a computational fundamental group (at least the critical points) nicely computable. We may thus use higher-dimensional persistent homology to distinguish certain homotopic loops (geodesic circles). In general we expect to be able to extract more geometric information from persistence and potentially use results of \cite{AAF} to detect higher-dimensional isometrically embedded spheres.
\end{itemize}

\section*{Comments on this version of the manuscript}

This version of the manuscript has been completed in 2024, four years after the journal publication of the article. Aside from corrected typos and slightly improved disposition, it contains a small mathematical correction: The $l$-lassos $\alpha * \beta * \alpha^-$ in Definition \ref{DefL} should be of finite length. In the journal version only $\beta$ was required to be of finite length, while in fact, $\alpha$ should be of finite length as well. In this version, the said error is corrected. All the main results remain as stated (with the same numbering). The change is only relevant in spaces, which are not SLSC (semi-locally simply connected), as for SLSC spaces both versions of Definition \ref{DefL} are equivalent. Accordingly, a minor modification (i.e., only the assumption/verification that the relevant paths and loops are of finite length) was required in the statement of Proposition \ref{PropIneq} and the proofs of Proposition \ref{PropIneq}, Proposition \ref{PropNull}, Proposition \ref{PropLength}, Definition \ref{DefGalpha}, Lemma \ref{LemmaConvergence}, Theorem \ref{PropFin}.


\end{document}